\documentclass[11pt,a4paper]{amsart}
\usepackage[utf8]{inputenc}
\usepackage[T1]{fontenc}
\usepackage[english]{babel}

\usepackage{amsfonts,amscd,latexsym}
\usepackage{amsaddr}
\usepackage[dvipsnames]{xcolor}
\usepackage{graphicx}
\usepackage{tikz}
\usetikzlibrary{arrows.meta, calc, intersections, backgrounds, positioning}
\usepackage{subcaption}
\definecolor{darkred}{rgb}{0.6, 0, 0}
\definecolor{violet}{rgb}{128, 0, 128}

\usepackage{pgfplots}
\pgfplotsset{compat=1.18}
\usepgfplotslibrary{groupplots}

\usepackage{url}
\usepackage{verbatim}

\usepackage[margin=0.95in]{geometry}
\usepackage[utf8]{inputenc}
\usepackage[scaled=.97,sups]{XCharter}
\usepackage[scaled=1.04,varqu,varl]{inconsolata}
\usepackage[type1]{cabin}
\usepackage[libertine,bigdelims,vvarbb,scaled=1.05]{newtxmath}
\usepackage[numbers,sort&compress]{natbib}
\usepackage{hyperref}

\theoremstyle{plain}
\newtheorem{theorem}{Theorem}[section]
\newtheorem*{mainNH}{Main Theorem on the NH system}

\newtheorem{lemma}[theorem]{Lemma}
\newtheorem*{triang}{Triangle Lemma}
\newtheorem*{barrier}{Barrier Lemma}
\newtheorem{corollary}[theorem]{Corollary}
\newtheorem{proposition}[theorem]{Proposition}
\newtheorem{conjecture}[theorem]{Conjecture}

\theoremstyle{definition}

\theoremstyle{remark}
\newtheorem{remark}[theorem]{Remark}

\DeclareMathOperator{\R}{\mathbb{R}}

\DeclareMathOperator{\Ph}{\mathbb{P}}
\DeclareMathOperator{\dist}{dist}

\DeclareMathOperator{\Dom}{Dom}

\newcommand{\ue}{\mathrm{e}}

\providecommand{\abs}[1]{\left\lvert#1\right\rvert}

\newcommand\T{\rule{0pt}{2.6ex}}       

\begin{document}

\title[Subthreshold oscillations and spiking in hybrid neuron models 
with a dynamic threshold]{Subthreshold oscillations and spiking 
in hybrid neuron models\\ with a dynamic threshold}

\author[P. Bart{\l}omiejczyk]{Piotr Bart{\l}omiejczyk}
\address{Piotr Bart{\l}omiejczyk,
Faculty of Applied Physics and 
Mathematics and BioTechMed Centre,\\
Gda\'nsk University of Technology,
Gabriela Narutowicza 11/12,\, 80-233 Gda{\'n}sk, Poland}
\email{piobartl@pg.edu.pl}

\author[J. Belmonte-Beitia]{Juan Belmonte-Beitia}
\address{Juan Belmonte-Beitia, Mathematical Oncology Laboratory (MOLAB), Ciudad Real,\linebreak
Laboratorio de Oncología Matemática, Instituto de Investigación Sanitaria de Castilla-La Mancha (IDISCAM), Toledo, 
\linebreak Institute of Applied Mathematics for Science and Engineering (IMACI), Ciudad Real, 
\linebreak Department of Mathematics, 
Escuela Técnica Superior de Ingeniería Industrial 
de Ciudad Real,
\linebreak University of Castilla-La Mancha 
Ciudad Real, 13071 Spain}
\email{juan.belmonte@uclm.es}

\author[J. Signerska-Rynkowska]{Justyna Signerska-Rynkowska}
\address{Justyna Signerska-Rynkowska,
Faculty of Applied Physics and 
Mathematics and BioTechMed Centre,
Gda\'nsk University of Technology, 
Gabriela Narutowicza 11/12,\, 80-233 Gda{\'n}sk, Poland}
\email{justyna.signerska@pg.edu.pl} 

\date{\today}



\begin{abstract} 
We investigate a class of two-dimensional hybrid non-autonomous
periodically forced neuron models known 
as the \emph{Meng--Huguet--Rinzel} neuron model,
with a \emph{dynamic threshold} and \emph{reset}
mechanism. We focus on the interplay between the continuous
subthreshold dynamics and discrete spike-induced resets. 
The model consists of a linear voltage equation coupled 
with a threshold variable governed by a nonlinear function 
of the membrane potential (voltage), 
and incorporates periodic external 
forcing in the form of either pulse 
or rectified sinusoidal currents. 
We analyze the existence and uniqueness of periodic solutions 
and prove that the non-hybrid version of
the system possesses a unique globally attracting periodic orbit. 
For the hybrid system, we show that there exists 
at most one periodic orbit without resets, 
while numerical simulations indicate the possibility 
of coexistence of a reset-free periodic orbit 
and periodic spiking attractors. 
We also prove that some natural polygons 
contained in the phase space of the hybrid system
form compact positively invariant globally attracting 
sets of this system. These results provide 
a rigorous mathematical framework for the analysis 
of dynamic threshold mechanisms in neuron models 
and contribute to the theoretical understanding 
of transitions between subthreshold oscillations 
and spiking behavior under a time-periodic forcing.
\end{abstract}

\maketitle

\begingroup
\renewcommand{\thefootnote}{}
\footnotetext{ORCID IDs: Piotr Bart{\l}omiejczyk 0000-0001-5779-4428; 
Juan Belmonte-Beitia 0000-0002-5003-1150; 
Justyna Signerska-Rynkowska 0000-0002-9704-0425.}
\footnotetext{Corresponding author: 
Justyna Signerska-Rynkowska (\url{justyna.signerska@pg.edu.pl}).}
\endgroup

\section*{Introduction}

The mathematical modeling of single-neuron dynamics
has a long history extending back to the early twentieth century.
One of the first quantitative descriptions was introduced
by Lapicque, who formulated the integrate-and-fire concept
as a simplified representation of neuronal excitation
\cite{Lapicque1907}. In this approach, the neuron is represented
by an electrical circuit whose state evolves in response
to incoming stimuli until a prescribed threshold is reached.
A spike is then generated and the state variable is reset.
Although highly simplified, this formulation initiated
a broad class of threshold-based models that continue
to play an important role in mathematical and computational neuroscience.

A fundamental development in biophysical neuron modeling
was the Hodgkin--Huxley model, formulated as a system
of nonlinear ordinary differential equations describing
the membrane potential together with voltage-dependent ionic currents
\cite{HodgkinHuxley1952}. This framework provided
a mechanistic description of action potential generation
and became a reference point for continuous-time neuron models.
Its subsequent reductions, most notably the FitzHugh--Nagumo model
\cite{FitzHugh1961,Nagumo1962}, retained the main qualitative
features of neuronal excitability while reducing the dimensionality
of the system and thereby facilitating mathematical analysis.
Further ODE-based models extended this approach toward
more complex firing regimes. An important example is
the Hindmarsh--Rose system, which reproduces both spiking
and bursting dynamics within a low-dimensional system
of coupled differential equations \cite{HindmarshRose1984}.
Together, these models illustrate how continuous-time descriptions
can represent neuronal dynamics at different levels
of biophysical detail and mathematical complexity.

Alongside continuous-time formulations, discrete-time models,
also referred to as map-based models,
have become an important alternative, particularly in the form
of low-dimensional iterated maps. Examples include the models
proposed by Chialvo \cite{Chialvo1995,Trujillo2023},
Rulkov \cite{Rulkov2001}, and Courbage et al.
\cite{Courbage2007,Courbage2010,Bartlomiejczyk2023,
Bartlomiejczyk2024,Bartlomiejczyk2025}.
Such models can reproduce characteristic spiking and bursting
patterns using comparatively simple recurrence relations.
Their low dimensionality and discrete update rules are advantageous
both for the investigation of nonlinear dynamical phenomena
and for simulations involving large populations of interacting neurons.

A third class of neuron models combines continuous evolution
with discrete events and can therefore be described
within a hybrid dynamical framework. A representative example
is the Izhikevich model, in which continuous equations governing
the evolution between spikes are supplemented by a discrete reset
applied after the firing condition is satisfied
\cite{Izhikevich2003}. This construction allows a variety
of neuronal firing patterns to be reproduced without introducing
the full complexity of detailed conductance-based models.
Related hybrid mechanisms also arise in models involving
a \emph{dynamic threshold}, where continuous evolution
is accompanied by switching or reset rules
\cite{MHR2012}. These formulations provide an intermediate
description between purely continuous and purely discrete models
and offer a flexible framework for studying neuronal dynamics.

Biological neurons can produce very diverse spiking patterns 
in response to the same transient or persistent inputs. 
For example, while many neurons fire repetitively 
for constant inputs of sufficiently large amplitude, 
some neurons fire only one or a few
spikes at the onset of any steady stimulus. 
These two behaviors are referred to, respectively, 
as tonic and phasic spiking and are connected with Hodgkin’s 
electrophysiological classification of neurons into Type I, 
Type II, and Type III excitability, 
based on their firing behavior under sustained current 
injection \cite{Hodgkin1948}. Within this framework, 
both Type I and Type II neurons fire repetitively 
when adequately excited with a constant stimulus, 
but they differ in their frequency-amplitude ($f$-$I$) curves. 
Type I neurons can sustain extremely low firing rates 
in response to weak currents; theoretically, their $f$-$I$ 
curve is continuous, with the firing frequency smoothly 
approaching zero as the current decreases toward threshold. 
By contrast, Type II neurons cannot maintain arbitrarily 
slow firing as their $f$-$I$ curve displays 
a sudden jump from zero to a distinctly non-zero frequency 
as soon as firing begins. These two excitability types 
have been studied extensively and, in the modeling 
of neuronal activity, have been linked, respectively, 
with a saddle-node on invariant circle (SNIC) 
bifurcation and an Andronov-Hopf bifurcation 
emerging at the onset of firing.

On the other hand, phasic or Type III neurons may generate 
transient spiking in response to an injected current, 
but are incapable of continuous repetitive firing, regardless 
of the amplitude of a constant input. 
Moreover, they exhibit markedly distinct computational properties 
and spiking patterns in response to time-varying 
inputs, which are connected with their high selectivity 
in encoding the occurrence and timing of rapid changes 
in the stimulus. In particular, they are more likely to display 
properties such as coincidence detection, 
post-inhibitory facilitation, slope detection,
or phase precision (see \cite{MHR2012,RSRT2021,RSRT2025}). 
Examples include some spinal cord neurons and auditory 
brain stem neurons in the medial superior olive (MSO neurons), 
which are capable of detecting time differences in the arrival 
of inputs originating from both ears on the order 
of tens of microseconds.

The work \cite{MHR2012} of Meng, Huguet, and Rinzel 
introduces an idealized model for Type III excitability 
in the form of a Leaky Integrate-and-Fire (LIF) model 
with a dynamic threshold, which can rise 
with the subthreshold voltage $V,$ replacing 
the conductance mechanism of the biophysical 
$8$-dimensional MSO neuron model developed by Rothman-Manis 
\cite{RothmanManis2003}. This model, which we call the MHR model, 
is the primary model studied in the current work,
and we introduce it in detail in Section~\ref{sec:model}. 
The authors of \cite{MHR2012} investigate this model 
with the input being the superposition of two half-wave 
rectified sinusoids with different phase shifts 
in relation to a sound wave arriving at each 
of the ears with a different delay. By distinguishing 
tonic and phasic regimes, they report that the 
tonic model has poor selectivity compared to the phasic model 
and that Type III excitability explains the extremely 
precise temporal computations of the MSO neurons 
in the auditory brain stem.

Granados and Huguet have made substantial contributions 
to the study of the MHR and related models 
with a dynamical threshold. In \cite{GH2019} 
they primarily focus on a general class of two-dimensional 
hybrid systems possessing an attracting equilibrium
in the absence of an input, 
and they explore their response to periodic square-wave driving, 
showing that such forcing induces periodic orbits prone to smooth 
and non-smooth grazing bifurcations. Next, by applying 
a 2D piecewise-smooth discontinuous stroboscopic map to the MHR model,
they develop a semi-analytical framework 
for identifying periodic orbits. Specifically, 
they identify parameter ranges where the map functions 
as a quasi-contraction, establishing the existence 
of a (locally) unique maximin periodic orbit. 
Using advanced numerical techniques, 
they further investigate the system's dynamical 
landscape, in particular highlighting regions 
of parameter space defined by coexisting 
stable orbits of varying periods.

The works of Brette and Platkiewicz (\cite{Brette2010,Brette2011}) 
also address the complexity of threshold variability
and its dependence on various factors, which reveal
the importance of handling these issues in the modeling 
of neuronal activity. While the all-or-none principle states 
that neurons only fire a full action potential once 
a specific excitation level is reached, the exact spike initiation 
threshold varies considerably across different cells, recording locations, 
and input stimuli. The value of this threshold is critical because 
it governs a neuron's firing rate and fundamental computational processes, 
such as coincidence detection. Traditionally, since the studies 
of Hodgkin and Huxley, the threshold in neurons has been viewed 
as a fixed voltage level at which
voltage-gated sodium channels open. 
However, as pointed out in \cite{Brette2010}, 
the concept of a rigid threshold has recently been challenged. 
Current research highlights ongoing debates regarding 
the source of threshold fluctuations in vivo and demonstrates 
that spike generation is influenced by complex properties 
of incoming signals, rather than by the membrane potential alone.

To account for this dynamic behavior, 
the authors of \cite{Brette2010,Brette2011}
propose the inactivating leaky integrate-and-fire (iLIF) model, 
which enhances standard integrate-and-fire mechanics by incorporating 
an adaptive threshold controlled by a piecewise-smooth
differential equation. Following an action potential, 
the membrane voltage resets to baseline, while the dynamic threshold 
shifts upward. Despite its mathematical simplicity, 
this phenomenological model accurately captures both 
the abrupt onset of spikes and the observed variability 
in initiation thresholds.

The above-mentioned works involving dynamic threshold models 
use phase-plane methods or computational tools. 
In the current work, we address
a few fundamental questions about the MHR model,
which we try to approach with full mathematical rigor. 
Therefore, we consider basic periodic stimuli, 
such as a rectified sine or pulse wave,
instead of a superposition of various periodic inputs,
and first focus on the non-hybrid non-autonomous 
system, i.e., the continuous version of the model 
without a resetting mechanism. The questions we address concern 
distinguishing between tonic and phasic spiking, 
the existence and uniqueness of spiking and 
non-spiking periodic solutions, and the
coexistence of various spiking patterns.

Developing a general pipeline for analyzing hybrid system dynamics 
is difficult, even for autonomous systems (i.e.,
systems with constant input), mainly because of reset-induced
discontinuities. A common approach is to use an impact map 
(also known as a firing phase or adaptation map, see,
e.g., \cite{TouboulBrette2009,CoombesThulWedgwood2012,COS2001}) 
whose iterates recover the sequence of consecutive spikes. 
However, this method suffers from domain restrictions 
and fails to account for trajectories that never hit the threshold. 
The analysis becomes even more complex with periodic inputs, 
as non-autonomous one-dimensional integrate-and-fire systems 
already exhibit highly complex dynamics 
\cite{Coombes1999,Keener1981,GKC2014,GK2015}.

We believe that our tools can be successfully adapted 
to the rigorous investigation of other hybrid models 
with a dynamic threshold. Therefore, the methods 
developed by us contribute to building a mathematical framework 
for the thorough analysis of dynamic threshold models
extending beyond the MHR model.

The organization of the paper is as follows. 
In Section~\ref{sec:prel}, we recall the necessary preliminaries. 
Section~\ref{sec:model} introduces the model under consideration,
i.e., the Meng--Huguet--Rinzel neuron model. 
Section~\ref{sec:dynnh} is devoted to the dynamics of 
the non-hybrid version of the system. 
The main results of the paper are presented 
in Section~\ref{sec:dynmhr}, where we investigate the dynamics 
of the MHR system. Finally, Section~\ref{sec:discussion} concludes 
the paper with a discussion of the obtained results 
and possible directions for future research.

\section{Preliminaries}\label{sec:prel}

\subsection{Global existence and uniqueness}
The hybrid model considered in this paper
is based on the system of the non-autonomous ODE
on the plane, which can be presented
in the following form
\begin{align}
    \dot{v} &=-v+J(t), \label{eq:ode1}\\
    \dot{\theta} &=g(v)+C\theta\label{eq:ode2},
\end{align}
where $C$ is a constant, $g$ is continuous (even smooth)
and $J(t)$ is either continuous or bounded, piecewise 
continuous with discrete set of discontinuities
(like step function).
Note that in the second case, the system  does not form
 a classical non-autonomous ODE, but rather 
the non-autonomous ODE with a discontinuous 
righthand side
satisfying the Carath\'eodory condition
(see~\cite{Filippov1988} for details).
However, due to the simplicity of the system
\eqref{eq:ode1}-\eqref{eq:ode2}, 
it is possible to justify the \emph{global existence and
uniqueness of solutions} (global means here that
the solution exists for all times $t$).
Namely, since the equation \eqref{eq:ode1} 
is a first-order linear ODE with 
an integrable forcing term $J(t)$ and does not depend 
on the variable $\theta$, it can be solved exactly 
using the integrating factor. 
The general solution with initial condition \(v(t_0)=v_0\) has the form
\begin{equation}
v(t) = v(t;t_0,v_0) =
v_0 e^{-(t - t_0)} + \int_{t_0}^{t} e^{-(t - s)} J(s)\, ds.
\end{equation}
If $J(t)$ is integrable then the function \(v(t)\) given 
by the above formula is \emph{absolutely continuous}
and defined for all $t\in\R$. 
In fact, since, by assumption, discontinuities of \(J(t)\) 
are isolated, the solution \(v(t)\) is 
\emph{globally continuous} and \emph{piecewise}~\(C^1\)
(precisely,  \(C^1\) on any open subinterval where \(J\) 
is continuous).
Moreover, \(v\) has left and right derivatives at each jump point
(isolated discontinuity of \(J\))
and the derivative jump equals the jump of \(J\).
Next, substitute \(v(t)\) into the equation \eqref{eq:ode2}:
\[
\dot{\theta}=g\big(v(t)\big)+C\,\theta(t).
\]
This is a first-order linear ODE in \(\theta\) with a known 
forcing term \(g(v(t))\) being continuous with respect to $t$.
Again the equation \eqref{eq:ode2} can be solved using
the integrating factor and the solution $\theta(t)$ is now
\emph{globally} \(C^1\) and defined for all $t\in\R$.
This explains the global existence of solutions.
In turn, the uniqueness of solutions follows directly from
either the Picard–Lindel\"of theorem, if $J(t)$ is continuous,
or Theorem 2 (Ch. 1 Sec. 1 in ~\cite{Filippov1988}),
if $J(t)$ is summable. A detailed analysis of the explicit form 
of the solutions under certain additional assumptions
is set out in Appendix~\ref{app:A}.

Finally, in both cases the system
\eqref{eq:ode1}-\eqref{eq:ode2} induces the \emph{process}
(\emph{evolution operator})
\[
\psi(t;t_0,z_0),\quad
\psi(t_0;t_0,z_0)=z_0=(v_0,\theta_0)\in\R^2,
\]
which gives the state at time $t$ of the solution 
that equals $z_0$ at time $t_0$.
The process is defined on the whole $\R\times\R\times\R^2$
and satisfies the \emph{semigroup (cocycle) property}
\[
\psi(t; s, \psi(s; t_0, z_0)) = \psi(t; t_0, z_0)
\quad\text{for all \(t \ge s \ge t_0.\)}
\]
Moreover, it is continuous but nonsmooth at points where 
the forcing term $J(t)$ has a jump. We will reserve the symbol 
$\phi(t;t_0,z_0)$ for future use to denote the process 
corresponding to the hybrid system, i.e., the system
\eqref{eq:ode1}-\eqref{eq:ode2} with additional 
reset conditions.

\subsection{Invariant sets}
Consider the IVP
\begin{equation}\label{eq:ivp}
\Bigg\{
\begin{aligned}
&\;\dot{x} = F(t,x) \\
&\;x(t_0) = x_0
\end{aligned}
\Bigg.
\end{equation}

Recall that $D\subset\R^n$ is called
\emph{positively invariant} for \eqref{eq:ivp}
if assuming $x_0\in D$
we have $x(t;t_0,x_0)\in D$
for all $t\ge t_0$.
This definition will also work
for the hybrid systems, which will be defined 
in the next section.
Roughly speaking, any trajectory of such a system which enters 
$D$ remains there forever.
The vector $y\neq0\in\R^n$ is called \emph{outer normal}
of $D$ in $x\in\partial D$ if
$B(x+y,\abs{y})\cap D=\emptyset$.
The following result is a classical consequence 
of Nagumo’s theorem (see e.g.~\cite{AubinCellina1984}).

\begin{proposition}\label{prop:posinv}
Assume that $D\subset\R^n$ is a closed convex set.
Then the following statements are equivalent:
\begin{enumerate}
    \item $D$ is positively invariant for \eqref{eq:ivp},
    \item $\langle F(t,x)\mid y\rangle\le0$
    for all $t$, $x\in\partial D$ and
    $y$ outer normal of $D$ in $x$.
\end{enumerate}
\end{proposition}

\section{Model}\label{sec:model}

\subsection{Introduction of the model}
\begin{figure}[htbp]
  \centering
  \begin{tikzpicture}[scale=1]
    \path[use as bounding box] (-1,-1) rectangle (4.8,4.8);

    \fill[blue!20]
      (0,0) -- (0,4.5) -- (4.5,4.5) -- cycle;

    \fill[red!20]
      (0,0) -- (4.5,0) -- (4.5,4.5) -- cycle;

    \draw[very thick, blue]
      (0,0) -- (4.5,4.5);

    \draw[->, very thick]
      (-1,0) -- (4.8,0)
      node[anchor=north west] {$v$};

    \draw[->, very thick, blue]
      (0,0) -- (0,4.8)
      node[anchor=south east, black] {$\theta$};

    \draw[very thick]
      (0,-1) -- (0,0);

    \node at (2.8,1)
      {\textcolor{red!70!black}{\textbf{spiking region $\mathbb{S}$}}};

    \node at (2.8,0.48)
      {\textcolor{red!70!black}{\textbf{(forbidden)}}};

    \node at (1.5,3.2)
      {\textcolor{blue!70!black}{\textbf{phase space} $\Ph$}};
  \end{tikzpicture}

  \caption{Phase space $\Ph$ of the hybrid system
    \eqref{eq:mhr1}--\eqref{eq:mhr3}. Resetting conditions~\eqref{eq:mhr3}, together with the assumption
    $\dot{v}>0$ for $v=0$ guarantee that trajectories of the hybrid
    system never leave the set $\Ph$.}
  \label{fig:phase}
\end{figure}
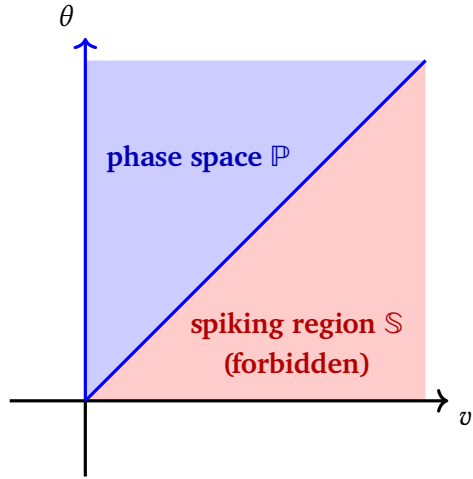

Let us consider a \emph{hybrid} model with two variables:
a \emph{dynamic threshold} $\theta$ and \emph{subthreshold voltage} $v$.
In this model a threshold $\theta$ can change
with the voltage $v$.
We will assume that the phase space of our hybrid system
is the subset of $\R^2$ determined by the double inequality
$\theta\ge v\ge0$ (see Figure~\ref{fig:phase})
and will be denoted by $\Ph$.
Moreover, let us consider the \emph{threshold line}
$\mathbb{L}=\{(v,\theta)\mid \theta=v\ge0\}$,
the \emph{subthreshold domain}
$\mathbb{D}=\{(v,\theta)\mid \theta>v\ge0\}=\Ph\setminus\mathbb{L}$
and the \emph{spiking region}
$\mathbb{S}=\{(v,\theta)\mid v>\theta\ge0\}$.
The equations of the hybrid non-autonomous system are:
\begin{align}
    \dot{v} &=-v+v_{\mathrm{rest}}+I(t), \label{eq:mhr1}\\
    \dot{\theta} &=\big(f(v)-\theta\big)/\tau \label{eq:mhr2},
\end{align}
with the \emph{resetting conditions}:
\begin{equation}
\text{
if $h(s)$ changes sign at $s=t$ then 
$v(t^+)=v_{\text{reset}}=0$ and 
$\theta(t^+)=\theta(t^-)+\Delta$,
} \label{eq:mhr3}  
\end{equation}
where $h(s)=\theta(s)-v(s)$
in the neighborhood of point $t$ is calculated 
using the solutions of the non-hybrid system 
\eqref{eq:mhr1}-\eqref{eq:mhr2}. 
Of course, since $h$ is continuous, we have $h(t)=0$, i.e., $v(t^-)=\theta(t^-)$,
but this condition is only necessary (not sufficient)
for the reset, because we do not want to apply the reset
if the contact with the diagonal (threshold line)
is internally tangential, i.e., \emph{grazing} happens
(see Figure~\ref{fig:grazing}).
This means that whenever a trajectory of the system
\eqref{eq:mhr1}-\eqref{eq:mhr2} crosses 
from left to right the threshold line
$\mathbb{L}$ (i.e., enters the spiking region)
the variables $v$ and $\theta$
are updated to new values given by \eqref{eq:mhr3},
which describes post-spike
resetting after which we record one spike.

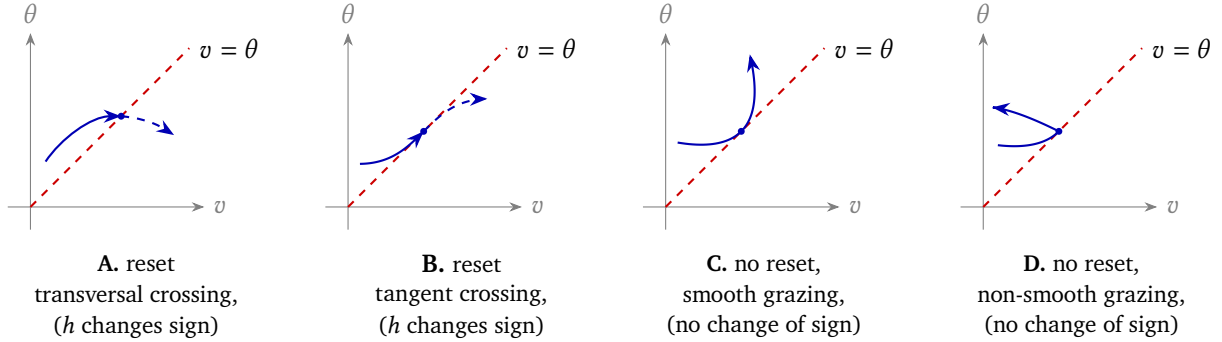
\begin{figure}[htbp]
  \centering
  \begin{tikzpicture}[
    >=Stealth,
    scale=1,
    traj/.style={thick, blue!70!black},
    postspike/.style={thick, blue!70!black, dashed},
    threshold/.style={red!80!black, dashed, thick},
    axis/.style={->, thin, gray},
    every node/.style={font=\small}
  ]

    \def\panelsep{4.2}

    \begin{scope}[shift={(0,0)}]
      \draw[axis] (-0.3,0) -- (2.3,0) node[right] {$v$};
      \draw[axis] (0,-0.3) -- (0,2.3) node[above] {$\theta$};
      \draw[threshold] (0,0) -- (2.1,2.1) node[right, black] {$v=\theta$};

      \draw[traj, ->]
        (0.2,0.6) .. controls (0.4,0.9) and (0.8,1.2) .. (1.2,1.2);

      \fill[blue!70!black] (1.2,1.2) circle (1.3pt);

      \draw[postspike, ->]
        (1.2,1.2) .. controls (1.4,1.18) and (1.6,1.14) .. (1.9,0.95);

      \node[below] at (1.4,-0.5) {{\footnotesize \textbf{A.} reset}};
      \node[below] at (1.4,-0.9) {{\footnotesize transversal crossing,}};
      \node[below] at (1.4,-1.3) {{\footnotesize ($h$ changes sign)}};
    \end{scope}

    \begin{scope}[shift={(\panelsep,0)}, rotate around={45:(1,1)}]
      \begin{scope}[rotate around={-45:(1,1)}]
        \draw[axis] (-0.3,0) -- (2.3,0) node[right] {$v$};
        \draw[axis] (0,-0.3) -- (0,2.3) node[above] {$\theta$};
        \draw[threshold] (0,0) -- (2.1,2.1) node[right, black] {$v=\theta$};
      \end{scope}

      \draw[traj, domain=-0.9:0, smooth, variable=\x, ->]
        plot({1+\x},{1-0.4*(\x)^3});

      \draw[postspike, domain=0:0.9, smooth, variable=\x, ->]
        plot({1+\x},{1-0.4*(\x)^3});

      \fill[blue!70!black] (1,1) circle (1.3pt);

      \begin{scope}[rotate around={-45:(1,1)}]
        \node[below] at (1.5,-0.5) {\textbf{B.} {\footnotesize reset}};
        \node[below] at (1.5,-0.9) {{\footnotesize tangent crossing,}};
        \node[below] at (1.5,-1.3) {{\footnotesize ($h$ changes sign)}};
      \end{scope}
    \end{scope}

    \begin{scope}[shift={(2*\panelsep,0)}, rotate around={45:(1,1)}]
      \begin{scope}[rotate around={-45:(1,1)}]
        \draw[axis] (-0.3,0) -- (2.3,0) node[right] {$v$};
        \draw[axis] (0,-0.3) -- (0,2.3) node[above] {$\theta$};
        \draw[threshold] (0,0) -- (2.1,2.1) node[right, black] {$v=\theta$};
      \end{scope}

      \draw[traj, domain=-0.7:0.8, smooth, variable=\x, ->]
        plot({1+\x},{1+(\x)^2});

      \fill[blue!70!black] (1,1) circle (1.3pt);

      \begin{scope}[rotate around={-45:(1,1)}]
        \node[below] at (1.3,-0.5) {{\footnotesize \textbf{C.} no reset,}};
        \node[below] at (1.3,-0.9) {{\footnotesize smooth grazing,}};
        \node[below] at (1.3,-1.3) {{\footnotesize (no change of sign)}};
      \end{scope}
    \end{scope}

    \begin{scope}[shift={(3*\panelsep,0)}, rotate around={45:(1,1)}]
      \begin{scope}[rotate around={-45:(1,1)}]
        \draw[axis] (-0.3,0) -- (2.3,0) node[right] {$v$};
        \draw[axis] (0,-0.3) -- (0,2.3) node[above] {$\theta$};
        \draw[threshold] (0,0) -- (2.1,2.1) node[right, black] {$v=\theta$};
      \end{scope}

      \draw[traj, domain=-0.7:0, smooth, variable=\x]
        plot({1+\x},{1+0.9*(\x)^2});

      \draw[traj, ->, domain=0:0.4, smooth, variable=\x]
        plot({1-\x},{1-2.2*(\x)^2+3*\x});

      \fill[blue!70!black] (1,1) circle (1.3pt);

      \begin{scope}[rotate around={-45:(1,1)}]
        \node[below] at (1.3,-0.5) {{\footnotesize \textbf{D.} no reset,}};
        \node[below] at (1.3,-0.9) {{\footnotesize non-smooth grazing,}};
        \node[below] at (1.3,-1.3) {{\footnotesize (no change of sign)}};
      \end{scope}
    \end{scope}

  \end{tikzpicture}

  \caption{Any trajectory crossing (A. transversal or B. tangential)
    the threshold line triggers a reset. Grazing (C. smooth or D. non-smooth)
    does not trigger a reset. Dashed line shows the continuation of the
    trajectory in the non-hybrid system.}
  \label{fig:grazing}
\end{figure}

It is assumed that
$v_\mathrm{rest}>0$ (resting voltage), $\Delta>0$ (jump value in $\theta$
after the reset) and $\tau>1$ are constants.
The authors of \cite{MHR2012} have chosen the formula
\begin{equation}\label{eq:f}
f(v)=a+\exp(b(v-c))
\quad\text{with $a,b,c>0$}
\end{equation}
for the $\theta$-nullcline, i.e.,  
the curve $\theta=f(v)$.
We assume that the input current $I(t)$
is periodic with a period $T$ and continuous
or with a finite number of discontinuities in $[0,T]$.
Moreover, we require $I(t)$ to satisfy
\begin{equation}\label{eq:input}
0\le I(t)\le A\quad\text{for some $A\ge0$ and all $t$.}
\end{equation}
In fact, we will mainly consider two versions of input:
\emph{half-wave-rectified wave} and \emph{pulse wave}.
The first one is given by
\begin{equation}\label{eq:inputSinus}
I(t)=A\cdot\max{\{0,\sin(2\omega\pi t)\}},  
\end{equation}
where $A$ is the \emph{amplitude}
and $\omega=1/T$ is the \emph{frequency}.
The second one is given by
\begin{equation}\label{eq:inputSquare}
I(t)=\begin{cases}
  A  & \text{if}\ t \in ( nT , nT + dT ],  \\
  0 & \text{if}\ t \in (nT + dT , (n + 1) T ],
\end{cases}
\end{equation}
where $n\in\mathbb{N}$ and $d\in[0,1]$ is the \emph{duty cycle}
(for a \emph{square wave} we have $d=0.5$). 

\begin{figure}[htbp]
  \centering
  \begin{tikzpicture}[>=Stealth, scale=1]
    \begin{groupplot}[
      group style={group size=2 by 1, horizontal sep=2cm},
      width=7cm,
      height=5cm,
      grid=major,
      xlabel={$t$},
      ylabel={$I(t)$},
      ymin=-0.1,
      ymax=1.1,
      samples=400,
      ytick={0,1},
      yticklabels={$0$,$A$}
    ]

      \nextgroupplot[
        title={Half-wave rectified wave}
      ]
      \addplot[blue, thick, domain=0:4]
        {max(0, sin(deg(2*pi*1.0*x)))};

      \nextgroupplot[
        title={Pulse (square) wave}
      ]
      \addplot[red, thick, domain=0:4, samples=400]
        {(x - floor(x/0.5)*0.5 < 0.25 ? 1 : 0)};

    \end{groupplot}
  \end{tikzpicture}

  \caption{Two main types of the input $I(t)$
    in the hybrid and non-hybrid systems:
    half-wave-rectified wave (left panel) and pulse wave
    (right panel).}
  \label{fig:impulse}
\end{figure}
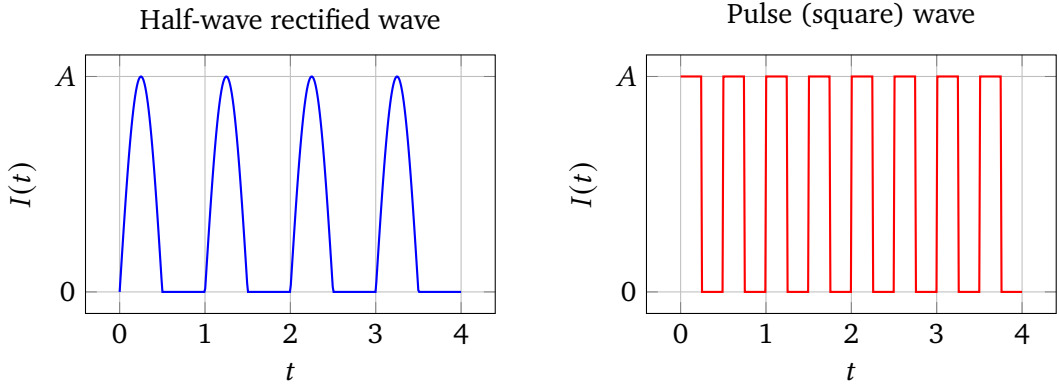

Finally, we will assume that, for $A = 0$, 
the system \eqref{eq:mhr1}-\eqref{eq:mhr2}
possesses an attracting equilibrium point inside $\Ph$,
which is equivalent to the condition
$f(v_\mathrm{rest})>v_\mathrm{rest}$.
Note, however, that we do not assume that
$f(v)>v$ for $v\ge0$ or even that 
$f(v_\mathrm{rest}+A)>v_\mathrm{rest}+A$.

\subsection{Local hybrid process}
The hybrid system \eqref{eq:mhr1}-\eqref{eq:mhr3} generates 
the \emph{local hybrid process}, which will be denoted by 
\[
\phi(t; t_0, x_0) = (\phi_v(t; t_0, x_0),\phi_\theta(t; t_0, x_0))
        =(\bar{v}(t; t_0, v_0), \bar{\theta}(t; t_0, \theta_0)).
 \]
Here, \(x_0=(v_0,\theta_0)\) denotes the initial state of 
the system at time \(t_0\), while \(\phi(t;t_0,x_0)\) represents 
the state of the hybrid system at time \(t\) obtained by starting 
from the initial condition \(x_0\). 
The term \emph{hybrid process} emphasizes that the evolution 
of the system combines two types of dynamics. 
Between reset events, the trajectory evolves continuously according 
to the differential equations \eqref{eq:mhr1}-\eqref{eq:mhr2}. 
Whenever the switching condition \(h(t)=\theta(t)-v(t)=0\) 
is reached and the function \(h\) changes sign, 
the trajectory undergoes an instantaneous jump determined 
by the reset conditions \eqref{eq:mhr3}. Therefore, the hybrid process 
\(\phi(t;t_0,x_0)\) describes both the continuous flow of the system 
and the discrete changes caused by threshold crossings.
Our local hybrid process has the following properties:
\begin{itemize}
			\item \emph{right continuity:} \(\phi(t; t_0, x_0)\) is 
            right-continuous and piecewise differentiable with respect to~$t$,
			\item \emph{discontinuity points:} at reset instants, 
            \(\bar{v}\) and \(\bar{\theta}\) satisfy the jump condition
            \eqref{eq:mhr3},
			\item \emph{semigroup property:} the hybrid process satisfies 
            the semigroup property, with the continuous flow between reset 
            events and the instantaneous state changes determined 
            by the reset conditions.
\end{itemize}
Note that some solutions of the hybrid system may not have 
the past, for example, the solutions that start from $\theta$-axis
and do not correspond to any reset in the hybrid system
and also the solutions that correspond to the crossing
the diagonal from the right to the left in the non-hybrid system
(see Figure~\ref{fig:orbits}). If we denote the domain
of $\phi$ as $\Dom(\phi)$ then we have
\begin{itemize}
                \item $\Dom(\phi)\subset
                \mathbb{R} \times \mathbb{R} \times \mathbb{P}$,
                \item for each $t_0\in\mathbb{R}$
                we have
                $
                [t_0,\infty)\times\{t_0\}\times\mathbb{P}\subset\Dom(\phi)
                $, i.e., the full forward trajectory always exists.
                Actually, there are two possibilities
                for $t_0\in\mathbb{R}$
                and $x_0\in\mathbb{P}$
                \[
                \text{either}\;\,
                \mathbb{R}\times\{t_0\}\times\{x_0\}\subset\Dom(\phi),
                \vspace{-1mm}
                \]
                \[
                \text{or}\;\,
                [t_0-r,\infty)\times\{t_0\}\times
                \{x_0\}\subset\Dom(\phi)
                \quad\text{for some $r\ge0$.}
                \]              
\end{itemize}

Finally, in our hybrid dynamical system, 
we can distinguish several types of orbits
depending on the presence of reset events. 
A \emph{reset-free (rf) orbit} is a trajectory 
that evolves entirely under the continuous flow 
and does not intersect the reset set. 
In contrast, a \emph{reset-including (ri) orbit} experiences 
at least one reset. Among ri orbits, we further distinguish 
\emph{finite-reset-including (fri) orbits}, 
which undergo only finitely many resets, 
and \emph{infinite-reset-including (iri) orbits}, 
which experience infinitely many resets. 
Observe that a periodic orbit can only be either rf or iri, 
since periodicity excludes the possibility 
of finitely many resets. 
Moreover, we will prove (see Proposition~\ref{prop:prelMHR}) that 
if our hybrid system (MHR) admits an fri orbit, 
then it necessarily admits an rf periodic orbit. 

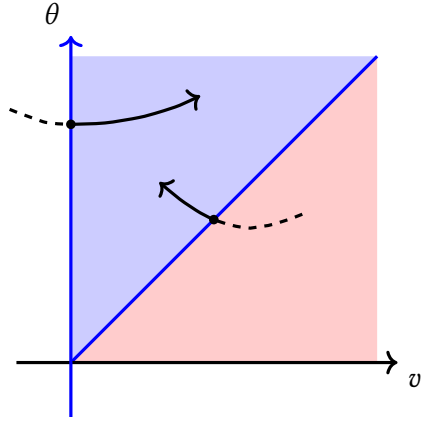
\begin{figure}[htbp]
  \centering
  \begin{tikzpicture}[scale=0.9]
    \draw[white] (-1,-1) rectangle (4.5,4.5);

    \fill[blue!20]
      (0,0) -- (0,4.5) -- (4.5,4.5) -- cycle;

    \fill[red!20]
      (0,0) -- (4.5,0) -- (4.5,4.5) -- cycle;

    \draw[very thick, blue]
      (0,0) -- (4.5,4.5);

    \draw[->, very thick]
      (-0.8,0) -- (4.8,0)
      node[anchor=north west] {$v$};

    \draw[->, very thick, blue]
      (0,-0.8) -- (0,4.8)
      node[anchor=south east, black] {$\theta$};

    \draw[very thick, black, dashed, smooth, tension=1]
      plot coordinates {
        (-0.9,3.72)
        (-0.5,3.57)
        (-0.3,3.52)
        (0,3.5)
      };

    \draw[very thick, black, ->, smooth]
      plot coordinates {
        (0,3.5)
        (0.5,3.52)
        (1.0,3.6)
        (1.3,3.68)
        (1.6,3.78)
        (1.9,3.92)
      };

    \draw[very thick, black, dashed, smooth, tension=1]
      plot coordinates {
        (3.4,2.18)
        (3.1,2.07)
        (2.8,2)
        (2.5,1.99)
        (2.1,2.1)
      };

    \fill[black] (2.1,2.1) circle (2pt);

    \draw[very thick, black, ->, smooth]
      plot coordinates {
        (2.1,2.1)
        (1.9,2.2)
        (1.6,2.4)
        (1.3,2.65)
      };

    \coordinate (P2) at (0,3.5);
    \fill[black] (P2) circle (2pt);
  \end{tikzpicture}

  \caption{Examples of trajectories without the past
    in the hybrid system (dashed lines represent parts of trajectories
    of the non-hybrid system).}
  \label{fig:orbits}
\end{figure}

\section{Dynamics of the non-hybrid system}\label{sec:dynnh} 
This section is devoted to the primary investigation
of the dynamics of the system of the classical non-autonomous ODE 
\eqref{eq:mhr1}-\eqref{eq:mhr2}
satisfying \eqref{eq:f} and \eqref{eq:input}.
We denote this system briefly as the NH system.
Clearly, the NH system provides a starting point
and natural background
for understanding the much more 
complicated dynamics of the MHR model.

In what follows, we will need
some auxiliary results.
To formulate them
we will use the following notation
\[
n_L=[-1,0],\quad
n_R=[1,0],\quad
n_B=[0,-1],\quad
n_T=[0,1]\quad
\text{and}\quad
F=[F_1,F_2],
\]
where
\[
F_1=F_1(t,v,\theta)=-v+v_{\text{rest}}+I(t)
\quad\text{and}\quad
F_2=F_2(t,v,\theta)=\big(f(v)-\theta\big)/\tau.
\]
The following result follows from the direct calculations
of the scalar product in $\R^2$.
\begin{lemma}\label{lem:ineq}
\small{
\[
\langle F\mid n_L\rangle=v-v_{\mathrm{rest}}-I(t),\quad
\langle F\mid n_R\rangle=-v+v_{\mathrm{rest}}+I(t),\quad
\langle F\mid n_B\rangle=\big(\theta-f(v)\big)/\tau,\quad
\langle F\mid n_T\rangle=\big(f(v)-\theta\big)/\tau.
\]}
\end{lemma}
Now assume that $\delta_1,\delta_2,\epsilon_1,\epsilon_2\ge0$
and set
\begin{align*}
    V_{\delta_1}&=v_\text{rest}-\delta_1,\\
    V^{\delta_2}&=v_\text{rest}+A+\delta_2,\\
    \Theta_{\delta_1,\epsilon_1}&=f(V_{\delta_1})-\epsilon_1,\\
    \Theta^{\delta_2,\epsilon_2}&=f(V^{\delta_2})+\epsilon_2,
\end{align*}
with $\delta_1$, $\delta_2$, $\epsilon_1$ and $\epsilon_2$ arbitrary nonnegative.
By definition and the monotonicity of $f$, 
\[
V_{\delta_1}\le v_\text{rest},\quad
V^{\delta_2}\ge v_\text{rest}+A\ge 0,\quad
\Theta_{\delta_1,\epsilon_1}\le f(v_\text{rest}),\quad
\Theta^{\delta_2,\epsilon_2}\ge f(v_\text{rest}+A)\ge 0.
\]
Consider the family of closed rectangles
\begin{equation}\label{eq:P}
P_{\delta_1,\epsilon_1}^{\delta_2,\epsilon_2}
=\{(v,\theta)\in\R^2 \mid 
V_{\delta_1}\leq v\leq V^{\delta_2}\text{ and } 
\Theta_{\delta_1,\epsilon_1}\leq \theta
\leq \Theta^{\delta_2,\epsilon_2} \}
\end{equation}
in the $v$-$\theta$ plane
(see Figure~\ref{fig:rec}).
Rectangles from that family will be called
\emph{admissible}.
Finally, let us single out the smallest 
rectangle in the family $P_\text{basic}=P_{0,0}^{0,0}$.

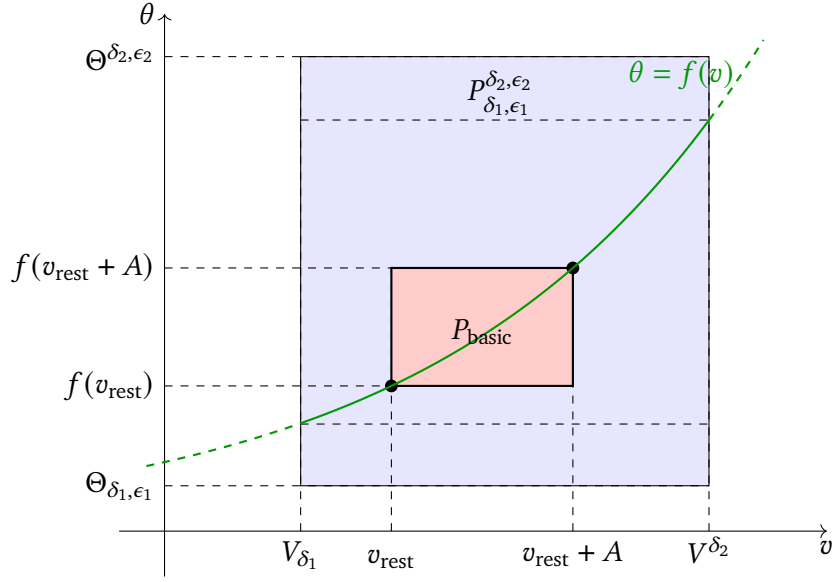
\begin{figure}[htbp]
  \centering
  \begin{tikzpicture}[scale=1.2]
    \def\vrest{2.5}
    \def\A{2}
    \def\frest{1.6}
    \def\fA{2.9}

    \def\Vleft{1.5}
    \def\Vright{6}

    \def\epsTwo{0.7}

    \def\corra{0.68}
    \def\corrb{4.03}

    \pgfmathsetmacro{\k}{ln(\fA/\frest)/\A}
    \pgfmathsetmacro{\fVright}{\frest*exp(\k*(\Vright-\vrest))}
    \pgfmathsetmacro{\Thetahigh}{\fVright+\epsTwo}

    \def\Thetalow{0.5}

    \draw[->]
      (-0.5,0) -- (7.3,0)
      node[below] {$v$};

    \draw[->]
      (0,-0.5) -- (0,5.7)
      node[left] {$\theta$};

    \draw[fill=blue!10]
      (\Vleft,\Thetalow)
      rectangle
      (\Vright,\Thetahigh);

    \draw[fill=red!20, thick]
      (\vrest,\frest)
      rectangle
      ({\vrest+\A},\fA);

    \draw[dashed]
      (\Vleft,0) -- (\Vleft,\Thetahigh);

    \draw[dashed]
      (\vrest,0) -- (\vrest,\fA);

    \draw[dashed]
      ({\vrest+\A},0) -- ({\vrest+\A},\fA);

    \draw[dashed]
      (\Vright,0) -- (\Vright,\Thetahigh);

    \draw[dashed]
      (0,\Thetalow) -- (\Vright,\Thetalow);

    \draw[dashed]
      (0,\frest) -- ({\vrest+\A},\frest);

    \draw[dashed]
      (0,\fA) -- ({\vrest+\A},\fA);

    \draw[dashed]
      (0,\Thetahigh) -- (\Vright,\Thetahigh);

    \draw[dashed]
      (\Vleft,\Thetalow+\corra)
      -- (\Vright,\Thetalow+\corra);

    \draw[dashed]
      (\Vleft,\Thetalow+\corrb)
      -- (\Vright,\Thetalow+\corrb);

    \fill
      (\vrest,\frest)
      circle (2pt);

    \fill
      ({\vrest+\A},\fA)
      circle (2pt);

    \draw[
      thick,
      dashed,
      green!60!black,
      domain=-0.2:\Vleft,
      smooth
    ]
      plot (\x,{\frest*exp(\k*(\x-\vrest))});

    \draw[
      thick,
      green!60!black,
      domain=\Vleft:\Vright,
      smooth
    ]
      plot (\x,{\frest*exp(\k*(\x-\vrest))});

    \draw[
      thick,
      dashed,
      green!60!black,
      domain=\Vright:6.6,
      smooth
    ]
      plot (\x,{\frest*exp(\k*(\x-\vrest))});

    \node[green!60!black]
      at (5.7,{\Thetahigh-0.2})
      {$\theta=f(v)$};

    \node[below]
      at (\Vleft,0)
      {$V_{\delta_1}$};

    \node[below]
      at (\vrest,-0.1)
      {$v_\text{rest}$};

    \node[below]
      at ({\vrest+\A},0)
      {$v_\text{rest}+A$};

    \node[below]
      at (\Vright,0.05)
      {$V^{\delta_2}$};

    \node[left]
      at (0,\Thetalow)
      {$\Theta_{\delta_1,\epsilon_1}$};

    \node[left]
      at (0,\frest)
      {$f(v_\text{rest})$};

    \node[left]
      at (0,\fA)
      {$f(v_\text{rest}+A)$};

    \node[left]
      at (0,\Thetahigh)
      {$\Theta^{\delta_2,\epsilon_2}$};

    \node
      at (3.5,2.2)
      {$P_{\text{basic}}$};

    \node
      at (3.7,{\Thetahigh-0.4})
      {$P_{\delta_1,\epsilon_1}^{\delta_2,\epsilon_2}$};
  \end{tikzpicture}

  \caption{A general positively invariant rectangle
    $P_{\delta_1,\epsilon_1}^{\delta_2,\epsilon_2}$ (in blue)
    and its smallest version $P_\text{basic}$ (in red)
    in the $(v,\theta)$-plane.}
  \label{fig:rec}
\end{figure}

\subsection{Basic and general positively invariant region}
The following result provides estimations
of positively invariant regions for
the NH system. 

\begin{proposition}\label{prop:naode}
For arbitrary
$\delta_1,\delta_2,\epsilon_1,\epsilon_2\ge0$
the closed rectangle 
$P_{\delta_1,\epsilon_1}^{\delta_2,\epsilon_2}$
is positively invariant with respect to the NH system. 
In particular, so is the rectangle
$P_\mathrm{basic}=P_{0,0}^{0,0}$. 
\end{proposition}

\begin{proof}
By Lemma~\ref{lem:ineq}, the monotonicity of $f$ 
and the estimation $0\le I(t)\le A$, we obtain
\begin{align*}
\langle F\mid n_L\rangle&=v-v_{\text{rest}}-I(t)\le0
\quad\text{for $v=V_{\delta_1}\le v_{\text{rest}}$ 
and every $\theta$},\\
\langle F\mid n_R\rangle&=-v+v_{\text{rest}}+I(t)\le0
\quad\text{for $v=V^{\delta_2}\ge v_{\text{rest}}+A$ 
and every $\theta$},\\
\langle F\mid n_B\rangle&=\big(\theta-f(v)\big)/\tau\le0
\quad\text{for every $v\ge V_{\delta_1}$ and 
$\theta=\Theta_{\delta_1,\epsilon_1}=f(V_{\delta_1})-\epsilon_1$},\\
\langle F\mid n_T\rangle&=\big(f(v)-\theta\big)/\tau\le0
\quad\text{for every $v\le V^{\delta_2}$ and 
$\theta=\Theta^{\delta_2,\epsilon_2}=f(V^{\delta_2})+\epsilon_2$}. 
\end{align*}
Now the assertion follows immediately from 
Proposition~\ref{prop:posinv} and 
the above inequalities. Note that in our case
(rectangle) we do not need to check the condition (2) of
Proposition~\ref{prop:posinv}  for all outer normal vectors.
It is enough to check it for the four outer normal vectors:
$n_L$, $n_R$, $n_B$ and $n_T$ on the respective sides of 
the rectangle $P_{\delta_1,\epsilon_1}^{\delta_2,\epsilon_2}$.
\end{proof}

\begin{remark}
Let us note that the rectangle 
$P_{\delta_1,\epsilon_1}^{\delta_2,\epsilon_2}$ 
is an enlargement of the rectangle $P_\mathrm{basic}$. 
However, it is not true that every rectangle containing 
$P_\mathrm{basic}$ is a positively invariant region. 
According to the definition of set 
$P_{\delta_1,\epsilon_1}^{\delta_2,\epsilon_2}$,
the rule for constructing invariant rectangles containing 
$P_\mathrm{basic}$ is as follows.
Take two values $v$, one such that $v_{1}\le v_{\mathrm{rest}}$ 
and another such that $v_{2}\ge v_{\mathrm{rest}} + A$. 
Then the rectangle whose sides are parallel to the coordinate axes 
and whose opposite vertices are 
$E = (v_{1}, f(v_{1}))$ and $F = (v_{2}, f(v_{2}))$
is a positively invariant region. Moreover, any rectangle obtained 
by arbitrarily extending this rectangle in the vertical direction 
is also a positively invariant region (see Figure~\ref{fig:rec}
for details).
\end{remark}

\begin{remark}
It is also worth noting that the analysis of positively invariant sets 
for the hybrid system will require consideration of rectangles 
$P_{\delta_1,\epsilon_1}^{\delta_2,\epsilon_2}$ of 
a special form, adapted on the one hand to the phase space of 
the hybrid system ($\theta\ge v\ge0$ ) and on the other 
to the reset phenomenon, which shifts the trajectory 
of the hybrid system from the threshold line ($\theta=v\ge0$)
to the reset line ($\theta\ge v=0$).
\end{remark}

\subsection{Global attracting region}
We will show that the rectangle $P_\mathrm{basic}$ 
is also a global attracting region for the NH system,
which means that, after a sufficiently long time, 
any orbit will be arbitrarily close 
to the rectangle $P_\mathrm{basic}$.
In what follows we will need the following simple result.

\begin{lemma}\label{lem:estim}
Let $L>0$ and $\rho(s)$ be an integrable function. Set
\[
H(t)=\ue^{-t/L}\int_{t_0}^t\!\!\! \rho(s)\ue^{s/L}\,ds
\]
Assume that there is $t^*\ge t_0$ such that
$\rho_{\min}\le \rho(s)\le \rho_{\max}$ for $s\ge t^*$. 
Then for every $\epsilon>0$ there is $\hat{t}\ge t^*$ 
such that for all $t\ge\hat{t}$ we get
\[
L\rho_{\min}-\epsilon\le
H(t)\le
L\rho_{\max}+\epsilon.
\]
\end{lemma}

\begin{proof}
Write $M=\int_{t_0}^{t^*}\! \rho(s)\ue^{s/L}\,ds$.
By assumption,
\begin{multline*}
H(t)\le
\ue^{-t/L}\int_{t_0}^{t^*}\!\!\! \rho(s)\ue^{s/L}\,ds+
\rho_{\max}\ue^{-t/L}\int_{t^*}^{t}\!\!\!\ue^{s/L}\,ds\le
\ue^{-t/L}\cdot M
+L\rho_{\max}\ue^{-t/L}\left(\ue^{t/L}-\ue^{t^*/L}\right)\\
\le L\rho_{\max}+\ue^{-t/L}\cdot M-\ue^{(t^*-t)/L}\cdot L\rho_{\max}\le
L\rho_{\max}+\epsilon
\end{multline*}
for $t$ large enough.
The justification of the second estimation is analogous.
\end{proof}

\begin{proposition}\label{prop:attract}
For every arbitrarily small $\gamma>0$ 
and every initial condition
$(t_0,v_0,\theta_0)$ from the extended phase space
there is $\overline{t}\ge t_0$ such that 
for all $t\ge\overline{t}$ we have
\[
v_\mathrm{rest}-\gamma<v(t)<v_\mathrm{rest}+A+\gamma
\quad\text{and}\quad
f(v_\mathrm{rest})-\gamma<\theta(t)<f(v_\mathrm{rest}+A)+\gamma,
\]
where $(v(t),\theta(t))$ is a solution with this initial condition. 
\end{proposition}

\begin{proof}
Fix $\gamma>0$ arbitrarily small.
We prove the assertion in two steps.
In the first step we prove the estimations for $v(t)$
and in the second for $\theta(t)$.

\noindent\emph{Step I.} Recall that from the general 
formula for solutions
\[
v(t)=v_\mathrm{rest}+
(v_0-v_\mathrm{rest})\ue^{t_0-t}+\ue^{-t}\int_{t_0}^t\!\!\! I(s)\ue^{s}ds,
\]
where $0\le I(s)\le A$. Hence, by Lemma~\ref{lem:estim} 
(taking $\epsilon=\gamma/2$), we obtain
\[
-\gamma/2\le
\ue^{-t}\int_{t_0}^t\!\!\! I(s)\ue^{s}ds\le
A+\gamma/2
\]
for $t$ greater than some $\hat{t}$.
Consequently, regardless of the sign of the term $v_0-v_\mathrm{rest}$,
\[
v_\mathrm{rest}-\gamma\le
v(t)\le
v_\mathrm{rest}+A+\gamma
\]
for $t\ge\overline{t}_1\ge\hat{t}\ge t_0$,
where $\overline{t}_1$ is large enough.\vspace{1mm}

\noindent\emph{Step II.}
Again from the general 
formula for solutions
\[
\theta(t)=
\ue^{(t_0-t)/\tau}\theta_0+
\ue^{-t/\tau}\frac1\tau
\int_{t_0}^t\!\! f\big(v(s)\big)\ue^{s/\tau}\,ds.
\]
By the monotonicity and continuity of $f$,
there is $0<\xi<\gamma/2$ such that
\[
f(v_\mathrm{rest})-\gamma/4\le
f(v_\mathrm{rest}-\xi)\le
f(v_\mathrm{rest}+A+\xi)\le
f(v_\mathrm{rest}+A)+\gamma/4.
\]
From the previous step for $s$ greater than some $t^*\ge t_0$
we have $v_\mathrm{rest}-\xi\le v(s)\le v_\mathrm{rest}+A+\xi$
and, in consequence, by monotonicity,
\[
f(v_\mathrm{rest})-\gamma/4\le
f(v(s))\le
f(v_\mathrm{rest}+A)+\gamma/4.
\]
Now from Lemma~\ref{lem:estim} 
(taking now $\epsilon=\tau\gamma/4$) we get
\[
f(v_\mathrm{rest})-\gamma/2\le
\ue^{-t/\tau}\frac1\tau
\int_{t_0}^t\!\! f\big(v(s)\big)\ue^{s/\tau}\,ds\le
f(v_\mathrm{rest}+A)+\gamma/2
\]
and finally
\[
f(v_\mathrm{rest})-\gamma\le
\theta(t)\le
f(v_\mathrm{rest}+A)+\gamma
\]
for $t\ge\overline{t}_2\ge t^*\ge t_0$,
where $\overline{t}_2$ is large enough.
Setting $\overline{t}=\max\{\overline{t}_1,\overline{t}_2\}$ 
completes the proof.
\end{proof}

\begin{remark}
Roughly speaking, Proposition~\ref{prop:attract} states
that the interesting dynamics of the NH system
is concentrated on the set $P_\mathrm{basic}$.
\end{remark}

\subsection{Stable periodic orbit}
Now we prove that the NH system has a unique
$T$-periodic orbit,
which is asymptotically stable.
In our analysis, we will need the following simple lemma.
\begin{lemma}\label{lem:scalarode}
Assume that $\alpha<0$ and $\beta(t)$ is $T$-periodic piecewise 
continuous (or even only integrable). 
Then the scalar differential equation
\[
\dot{x}=\alpha x+\beta(t)
\]
has a unique $T$-periodic orbit,
which is asymptotically stable.
\end{lemma}

\begin{proof}
Since we have the explicit formula for a solution
with initial condition $x(0)=x_0$, namely,
\[
x(t)=e^{\alpha t}\Bigl(x_0+\int_{0}^{t} e^{-\alpha s} \beta(s)\,ds\Bigr),
\]
we have also the explicit formula for the scalar
Poincar\'e map $\Pi\colon\mathbb{R}\to\mathbb{R}$
\[
\Pi(x)=e^{\alpha T}x+\gamma,
\]
where $\gamma$ is constant.
Now, since, by assumption, $\Pi$ is a linear contraction,
it has a unique fixed point,
which attracts all points from the real line.
This fixed point corresponds to
a unique $T$-periodic orbit of the differential equation,
which is asymptotically stable.
\end{proof}

The following result describes the periodicity aspect
of the non-hybrid system.

\begin{theorem}\label{thm:NHperiodic}
The NH system has a unique $T$-periodic orbit,
which is asymptotically stable.
\end{theorem}

\begin{proof}
Let us conduct our analysis in two steps
corresponding to the equations \eqref{eq:mhr1} and 
\eqref{eq:mhr2}.\vspace{0.5mm}

\noindent \emph{Step 1.}
Applying Lemma~\ref {lem:scalarode} to 
the scalar equation \eqref{eq:mhr1} 
we obtain a unique $T$-periodic orbit $v(t)$, 
which is asymptotically stable.\vspace{0.5mm}

\noindent \emph{Step 2.}
Now we insert the periodic solution $v(t)$ 
from Step 1 into the scalar equation \eqref{eq:mhr2}
and apply again Lemma~\ref {lem:scalarode} 
this time to the equation \eqref{eq:mhr2}
to obtain a unique $T$-periodic orbit $\theta(t)$, 
which is also asymptotically stable.\vspace{0.5mm}

\noindent Finally, the pair $(v(t),\theta(t))$ forms
a $T$-periodic planar orbit,
which is asymptotically stable. 
Of course, such an orbit is unique, because
its first coordinate must coincide with $v$ (by uniqueness of~$v$)
and, in consequence, its second coordinate 
must coincide with $\theta$ (by uniqueness of~$\theta$).
\end{proof}

\subsection{Global periodic attractor}

We start this subsection with an auxiliary lemma
concerning the convergence of solutions
of the NH system.
Roughly speaking, it states that any two solutions 
of the NH system converge, 
regardless of the time and place of departure.
Namely, consider two solutions $(v(t),\theta(t))$
and $(\overline{v}(t),\overline{\theta}(t))$ of the NH system,
the first with initial condition $(t_0,v_0,\theta_0)$
and the second with $(\overline{t}_0,\overline{v}_0,\overline{\theta}_0)$.

\begin{lemma}\label{lem:asymptotic}
For every $\epsilon>0$ there is $t^*\ge\max\{t_0,\overline{t}_0\}$
such that 
\[
\abs{v(t)-\overline{v}(t)}<\epsilon
\quad\text{and}\quad
\abs{\theta(t)-\overline{\theta}(t)}<\epsilon
\quad\text{for $t\ge t^*$.}
\]
\end{lemma}

\begin{proof}
Let $\epsilon>0$ be arbitrary.
Without loss of generality we can assume that
$\overline{t}_0\ge t_0$. From the general
formulas for the solutions (see Subsection~\ref{appsub:general}
for details) we obtain 
{\small
\begin{align}
v(t)-\overline{v}(t)&=\ue^{-t}
\bigg[(v_0-v_{\mathrm{rest}})\ue^{t_0}-
(\overline{v}_0-v_{\mathrm{rest}})\ue^{\overline{t}_0}+
\int_{t_0}^{\overline{t}_0}\!\!\!\!
I(s)\ue^s\,ds
\bigg]\\
\theta(t)-\overline{\theta}(t)&=\ue^{-t/\tau}
\bigg[\theta_0\ue^{t_0/\tau}-\overline{\theta}_0\ue^{\overline{t}_0/\tau}+
\frac1\tau\int_{t_0}^{\overline{t}_0}\!\!\!
f(v(s))\ue^{s/\tau}\,ds+
\frac1\tau\int_{\overline{t}_0}^{t}\!\!\!
\big(f(v(s))-f(\overline{v}(s))\big)\ue^{s/\tau}\,ds
\bigg]
\end{align}}
Note that the terms
\[
M_1=(v_0-v_{\mathrm{rest}})\ue^{t_0}-
(\overline{v}_0-v_{\mathrm{rest}})\ue^{\overline{t}_0}+
\int_{t_0}^{\overline{t}_0}\!\!\!\!
I(s)\ue^s\,ds
\quad\text{and}\quad
M_2=\theta_0\ue^{t_0/\tau}-\overline{\theta}_0\ue^{\overline{t}_0/\tau}+
\frac1\tau\int_{t_0}^{\overline{t}_0}\!\!\!
f(v(s))\ue^{s/\tau}\,ds
\]
are constant, i.e., independent of time $t$.
In consequence, obviously, there is $t_1\ge\overline{t}_0$
such that $\abs{v(t)-\overline{v}(t)}=\ue^{-t}\abs{M_1}<\epsilon$
for all $t\ge t_1$. Moreover, by Proposition~\ref{prop:attract}, 
there is $t_2\ge t_1$ such that
$v(s),\overline{v}(s)\in[v_{\mathrm{rest}}-1,v_{\mathrm{rest}}+A+1]$
for $s\ge t_2$. 

By the uniform continuity of $f$ on
$[v_{\mathrm{rest}}-1,v_{\mathrm{rest}}+A+1]$
(of course, instead of $1$, we can take any positive real number),
there is $\delta\in(0,\epsilon)$ such that if 
$v,\overline{v}\in[v_{\mathrm{rest}}-1,v_{\mathrm{rest}}+A+1]$
and $\abs{v-\overline{v}}<\delta$ then
$\abs{f(v)-f(\overline{v})}<\epsilon/2$.
Next, let us choose $t'\ge t_2$ such that 
$\abs{v(s)-\overline{v}(s)}<\delta$, which is possible
from the previous part of that proof. Now, denoting
$M_3=\frac1\tau\int_{\overline{t}_0}^{t'}\!\!\!
\big(f(v(s))-f(\overline{v}(s))\big)\ue^{s/\tau}\,ds$,
we obtain
{\small
\begin{multline*}
\abs{\frac1\tau\int_{\overline{t}_0}^{t}\!\!\!
\big(f(v(s))-f(\overline{v}(s))\big)\ue^{s/\tau}\,ds}\\
\le\abs{\frac1\tau\int_{\overline{t}_0}^{t'}\!\!\!
\big(f(v(s))-f(\overline{v}(s))\big)\ue^{s/\tau}\,ds}+
\abs{\frac1\tau\int_{t'}^{t}\!\!\!
\big(f(v(s))-f(\overline{v}(s))\big)\ue^{s/\tau}\,ds}
\\\le\abs{M_3}+\frac{\epsilon}{2}[\ue^{t/\tau}-\ue^{t'/\tau}]
\le\abs{M_3}+\frac{\epsilon}{2}\ue^{t/\tau}.
\end{multline*}}
Finally, we see at once that there is 
$t^*\ge t'\ge t_2\ge t_1$
such that for $t\ge t^*$ we have
\[
\abs{\theta(t)-\overline{\theta}(t)}\le
\ue^{-t/\tau}\left(\abs{M_2}+\abs{M_3}+\frac{\epsilon}{2}\ue^{t/\tau}\right)=
\ue^{-t/\tau}\left(\abs{M_2}+\abs{M_3}\right)+\frac{\epsilon}{2}
<\frac{\epsilon}{2}+\frac{\epsilon}{2}=\epsilon,
\]
which completes the proof.
\end{proof}

The following theorem contains the complete
qualitative description of the dynamics of the NH system.
\begin{mainNH}
\phantom{xxx}
\begin{enumerate}
    \item The NH system has exactly one periodic orbit which is 
    $T$-periodic, asymptotically stable and contained 
    in the rectangle $P_\mathrm{basic}$. In consequence,
    there are no other periodic orbits regardless 
    of their period in the NH system.   
    \item All other orbits of the NH system tend asymptotically
    (as $t\to\infty$) towards this special unique periodic orbit, 
    which consequently is a unique global attractor of the system.
\end{enumerate}
\end{mainNH}  

\begin{proof}
By Theorem~\ref{thm:NHperiodic},
the NH system has a unique $T$-periodic orbit,
which is asymptotically stable.  
Note that this periodic orbit must be contained
in the global attracting region $P_\mathrm{basic}$.
Otherwise, by periodicity, it would return to 
a point lying at fixed constant positive distance from 
the rectangle $P_\mathrm{basic}$, contradicting
Proposition~\ref{prop:attract}. 
Moreover, all orbits of the NH system 
tend asymptotically (as $t\to\infty$) towards 
this periodic orbit. To see this, treat $(v(t),\theta(t))$
as an arbitrary orbit and $(\overline{v}(t),\overline{\theta}(t))$ 
as this fixed periodic orbit and apply Lemma~\ref{lem:asymptotic}.
Finally, such an asymptotic behavior of all orbits obviously rules out 
the existence of other periodic orbits of the NH system.
\end{proof}

\begin{figure}[!hbtp]
    \centering
    \includegraphics[width=0.98\textwidth]{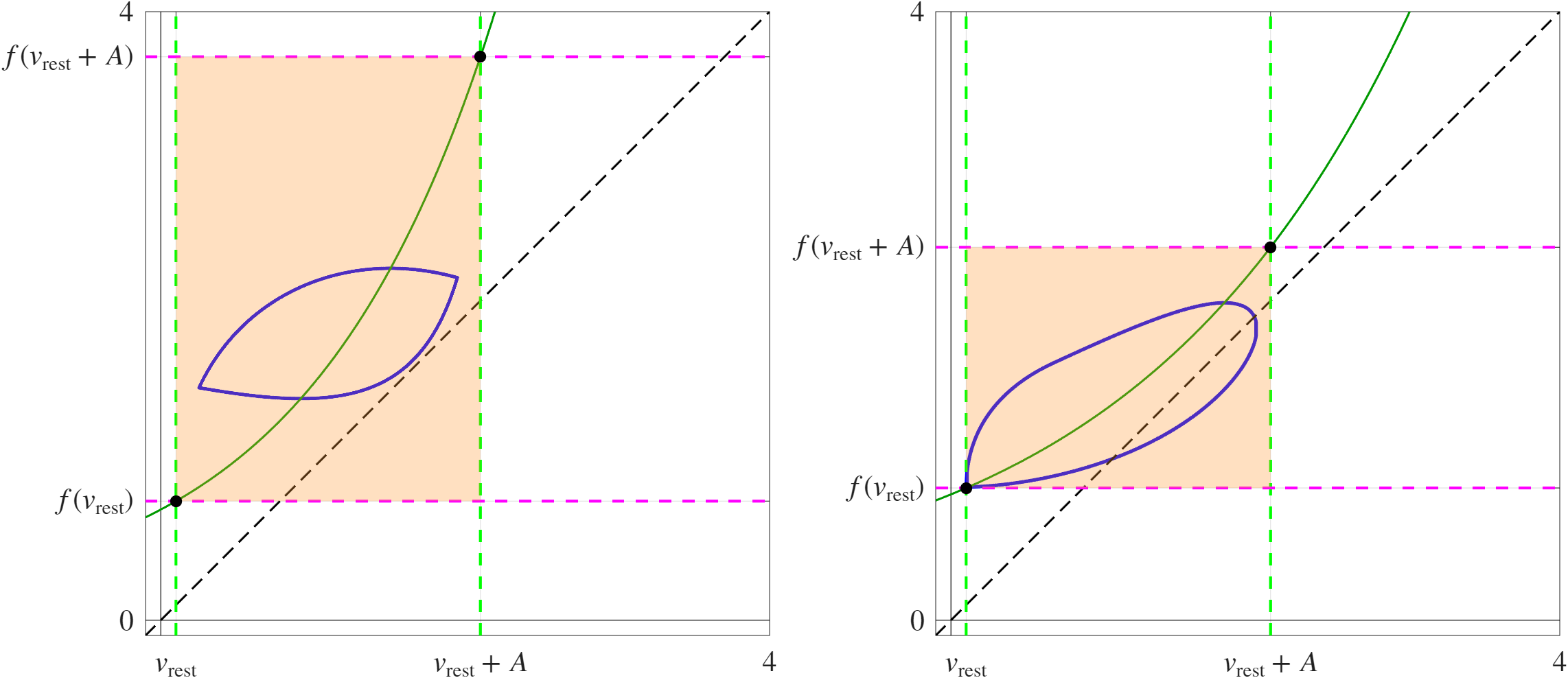}
    \caption{A unique global periodic attractor 
    (in blue) inside the positively invariant rectangle
    $P_\mathrm{basic}$ (in orange) in the NH system. 
    Left panel: square wave input, right panel: 
    half-wave-rectified wave input.
    Simulation parameters: common $a = 0.08, c = 0.53,
    v_\mathrm{rest} = 0.1, \tau = 2, A = 2$; 
    left panel $b = 0.82, T = 5, d = 0.5$, right 
    panel $b = 0.55, \omega = 0.05$.}
    \label{fig:twoNH}
\end{figure}

\begin{remark}
To put it simply, the NH system has the following two key properties:
\begin{itemize}
    \item it loses ``memory'' of the past, i.e.,
    forgets the initial conditions (both time and location); 
    in other words, all solutions ``merge'' into a single solution,
    \item it is globally stable with a unique periodic attractor (see Figure~\ref{fig:twoNH}).
\end{itemize}
\end{remark}

\section{Dynamics of the MHR system}\label{sec:dynmhr}
In this section we study the dynamics of the MHR system, i.e.,
the hybrid system \eqref{eq:mhr1}-\eqref{eq:mhr3}
satisfying the additional conditions 
\eqref{eq:f} and \eqref{eq:input}.
We start with basic observations, which follow
immediately from the properties of its non-hybrid version,
i.e., the NH system (Subsection~\ref{subsec:obvious}).
Then we present the description of the dynamics of the MHR system, 
which requires a more detailed mathematical analysis  
(Subsections~\ref{subsec:absencereset}-\ref{subsec:gradation}).

\subsection{Preliminary results on the MHR system}
\label{subsec:obvious}
The properties of the NH system proved in Section~\ref{sec:dynnh} 
have some obvious (but important) direct implications for 
the corresponding hybrid MHR system. 
We present them below.
\begin{proposition}\label{prop:prelMHR}
\phantom{xxx}    
\begin{enumerate}
    \item The MHR system has at most one reset-free periodic orbit.
    Moreover, if such a reset-free periodic orbit exists 
    in the MHR system then it is contained 
    in the intersection of the phase space 
    $\Ph$ and the rectangle $P_\mathrm{basic}$, i.e.,
    in the polygon $\Ph\cap P_\mathrm{basic}$.
    \item There is a reset-free periodic orbit 
    in the MHR system if and only if there is an orbit 
    with finite (possibly zero) number of resets. In other words,
    there is no reset-free periodic orbit if and only if
    all orbits have infinite number of resets.
    Moreover, if there is a reset-free periodic orbit,
    then all orbits with finite number of resets 
    tend (as $t\to\infty$) towards this periodic orbit.
\end{enumerate}
\end{proposition}
\begin{proof}
\emph{Ad (1)} It is evident that every reset-free periodic orbit
of the MHR system (if exists) originates from a periodic orbit
of the NH system, but, by Main Theorem on the NH system,
the NH system has a unique periodic orbit. Thus the MHR system 
has at most one reset-free periodic orbit. 
In addition, since the unique periodic orbit of the NH system 
is contained in $P_\mathrm{basic}$, the reset-free periodic 
orbit of the MHR system (if exists) is contained in
$\Ph\cap P_\mathrm{basic}$.\vspace{0.5mm}

\noindent\emph{Ad (2)} Note that a reset-free periodic orbit
is obviously an orbit with finite number of resets (namely $0$).
On the other hand, since the phase space $\Ph$ is closed 
in $\R^2$, the attractor of each orbit with finite number of resets
is, by Main Theorem on the NH system, the unique 
periodic orbit of the NH system contained in $\Ph$, i.e.,
a reset-free periodic orbit of the MHR system.
This completes the proof.
\end{proof}

To sum up, there are only two possibilities: 
\begin{enumerate}
\item 
either the unique periodic orbit of the NH system
is contained in the phase space $\Ph$,
and, in consequence, it becomes 
the unique reset-free periodic orbit of the MHR system
\item
or the unique periodic orbit of the NH system
is not contained in the phase space $\Ph$
and, in consequence, it is ``destroyed'' by
the reset phenomenon, which means that
there are no reset-free periodic orbits 
in the MHR system and all orbits have an infinite number of resets.
\end{enumerate}
However, if the MHR system has exactly one reset-free periodic orbit,
then orbits with an infinite number of resets may or may not occur.
Even more interestingly, numerical simulations presented
in Section~\ref{sec:discussion} suggest that in the second case 
at least one reset-including periodic orbit must arise 
in the MHR system, whereas in the first case 
such an orbit may or may not arise.

\subsection{The MHR system with no resets}
\label{subsec:absencereset}
In this subsection, we will give a sufficient condition 
for the complete absence of resets in the MHR system. 
Although such a situation is not typical in hybrid systems, 
for the sake of completeness let us deal with
this borderline case. Clearly, the occurrence of resets
in the system depends directly on the direction of 
the non-autonomous vector field at points on the 
diagonal. Namely, it is easy to check that 
there is no reset in the MHR system if 
at all points on the diagonal
for all $t$
the following condition is satisfied
\begin{equation}\label{eq:diag}
\langle F\mid n_D\rangle=
v-v_{\mathrm{rest}}-I(t)+\big(f(v)-v\big)/\tau
\ge0,
\end{equation}
where $n_D=[-1,1]$. Based on~\eqref{eq:diag}
we can provide the simple sufficient condition
for no resets in the system.
\begin{proposition}\label{prop:noresetcond}
Assume that the~hybrid system of the non-autonomous ODE 
\eqref{eq:mhr1}-\eqref{eq:mhr3}
satisfies \eqref{eq:f}, \eqref{eq:input}. 
If 
\[
a+\exp(-bc)\ge\tau(v_{\mathrm{rest}}+A)
\]
then there is no reset in the hybrid system 
under consideration.
\end{proposition}

\begin{proof}
Observe that~\eqref{eq:diag} is equivalent to
$(\tau-1)v+f(v)\ge\tau(v_{\mathrm{rest}}+I(t))$.
Now, by assumption,
\[
(\tau-1)v+f(v)\ge
f(0)=a+\exp(-bc)\ge
\tau(v_{\mathrm{rest}}+A)\ge
\tau(v_{\mathrm{rest}}+I(t)),
\]
which completes the proof.
\end{proof}

\subsection{Resets on the diagonal in the MHR system}
\label{subsec:diagreset}
Note that the condition~\eqref{eq:diag}
from the previous subsection also allows for a more detailed 
analysis of which parts of the diagonal $v=\theta$
cannot have resets and in which they may occur.
The details of this analysis are set out below
and in Figure~\ref{fig:reset}.
\begin{proposition}\label{prop:noreset}
Assume that $f(v)\ge v$ for $v\ge0$, i.e.,
the nullcline $\theta=f(v)$ lies above
the diagonal.
Then the MHR system has no resets 
at all points on the diagonal
$v=\theta$ for which $v\ge v_{\mathrm{rest}}+A$.
In consequence, resets in the MHR system
may appear only at points on the diagonal 
for which $0\le v< v_{\mathrm{rest}}+A$.
Moreover, assuming in addition that 
the input $I(t)=0$ at time $t$
there are no resets at time $t$ at the points 
on the diagonal satisfying
$v_{\mathrm{rest}}\le v\le v_{\mathrm{rest}}+A$.
\end{proposition}

\begin{proof} 
If $f(v)\ge v$ and $v\ge v_{\mathrm{rest}}+A$
then $\langle F\mid n_D\rangle\ge0$
and we are done. Similarly, if $I(t)=0$
then $v_{\mathrm{rest}}\le v\le v_{\mathrm{rest}}+A$
implies $\langle F\mid n_D\rangle\ge0$,
which is the desired conclusion.
\end{proof}

\begin{figure}[htbp]
  \centering
  \begin{tikzpicture}[scale=1.2]
    \def\vr{1}
    \def\A{1.5}
    \pgfmathsetmacro{\vra}{\vr+\A}

    \draw[->]
      (-0.8,0) -- (4.2,0)
      node[right] {$v$};

    \draw[->]
      (0,-0.7) -- (0,4.2)
      node[above] {$\theta$};

    \fill[cyan!30]
      (0,0) -- (0,4) -- (4,4) -- (4,0) -- cycle;

    \fill[white]
      (0,0) -- (4,4) -- (4,0) -- cycle;

    \fill[cyan!30]
      (0,0) -- (4,4) -- (0,4) -- cycle;

    \draw[dashed]
      (0,0) -- (4,4)
      node[above right] {$\theta=v$};

    \draw[very thick, blue, dotted]
      (0,0) -- (\vr,\vr);

    \draw[very thick, red, dashed]
      (\vr,\vr) -- (\vra,\vra);

    \draw[very thick, red]
      (\vra,\vra) -- (4,4);

    \draw[thick, dashed]
      (\vr,0) -- (\vr,4);

    \draw[thick, dashed]
      (\vra,0) -- (\vra,4);

    \filldraw[black]
      (\vr,0) circle (1.5pt)
      node[below]
      {$\phantom{AA}v=v_\mathrm{rest}\phantom{I}$};

    \filldraw[black]
      (\vra,0) circle (1.5pt)
      node[below]
      {$\phantom{AAA}v=v_\mathrm{rest}+A$};
  \end{tikzpicture}

  \caption{Reset analysis of the diagonal for
    the MHR system satisfying $f(v)\ge v$: solid red line =
    no reset at all, dashed red line = no reset for the input
    $I(t)=0$ and possible resets for $I(t)>0$, dotted blue line
    = possible resets. Phase space in blue.}
  \label{fig:reset}
\end{figure}
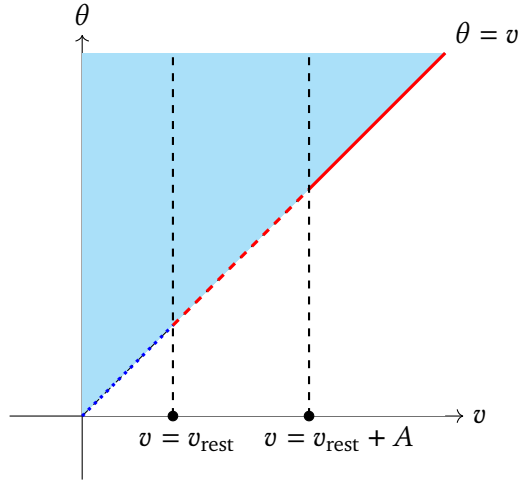

\subsection{Four types of positively invariant globally attracting 
regions (PIGARs) in the MHR system}

We say that a compact subset $Q$ of the phase space $\Ph$
is a \emph{positively invariant globally attracting region}
(\emph{PIGAR} for short) if it is positively invariant and
all orbits of the system converge towards $Q$, i.e.,
for every arbitrarily small $\gamma>0$ 
and every initial condition
$(v_0,\theta_0,t_0)$ from the extended phase space
there is $\overline{t}\ge t_0$ such that 
for all $t\ge\overline{t}$ we have
\[
\dist\big((v(t),\theta(t)),Q\big)<\gamma
\]
where $(v(t),\theta(t))$ is an orbit corresponding
to this initial condition.
Roughly speaking, a PIGAR is a compact subset 
of the phase space in which non-trivial, 
interesting dynamics is concentrated.
Below using the results concerning the NH system
from the previous section we will show
that some natural polygons (rectangles or pentagons,
see Figure~\ref{fig:pigars})
form PIGARs for the MHR system. 

\begin{figure}[htbp]
  \def\hAdd{0.8}
  \def\commonScale{0.6}
  \def\lblScale{1.0}
  \def\diagScale{1.0}

  \centering

  \begin{minipage}{0.3\textwidth}
    \centering
    \scriptsize
    \textbf{Case 1: $f(0) > v_{\mathrm{rest}} + A$} \\

    \begin{tikzpicture}[
      scale=\commonScale,
      dashed line/.style={dashed, thin, gray},
      dot/.style={circle, fill=black, inner sep=1pt}
    ]
      \def\vRA{3.0}
      \def\vRest{0.8}

      \pgfmathsetmacro{\a}{2.6}
      \pgfmathsetmacro{\b}{0.3}
      \pgfmathsetmacro{\c}{0.2}

      \pgfmathsetmacro{\yZero}{\a + exp(\b*(0 - \c))}
      \pgfmathsetmacro{\yR}{\a + exp(\b*(\vRest - \c))}
      \pgfmathsetmacro{\yRA}{\a + exp(\b*(\vRA - \c))}

      \begin{scope}[on background layer]
        \fill[blue!5]
          (0,0) -- (0,6.5) -- (5,6.5) -- (5,5) -- cycle;
      \end{scope}

      \begin{scope}
        \fill[orange!35]
          (\vRest,\yR) --
          (\vRA,\yR) --
          (\vRA,\yRA) --
          (\vRest,\yRA) --
          cycle;

        \draw[brown!60!black, ultra thick]
          (\vRest,\yR) --
          (\vRA,\yR) --
          (\vRA,\yRA) --
          (\vRest,\yRA) --
          cycle;

        \draw[dashed, thin, black!70]
          (0,\vRA) -- (\vRA,\vRA);

        \node[below left, scale=0.8]
          at (\vRest,\yR) {$X_1$};

        \node[below right, scale=0.8]
          at (\vRA,\yR) {$X_2$};

        \node[above right, scale=0.8]
          at (\vRA,\yRA) {$X_3$};

        \node[above left, scale=0.8]
          at (\vRest,\yRA) {$X_4$};
      \end{scope}

      \draw[->, >=Stealth]
        (-0.5,0) -- (5,0)
        node[right] {\scriptsize $v$};

      \draw[->, >=Stealth]
        (0,-0.5) -- (0,6.5)
        node[above] {\scriptsize $\theta$};

      \draw[dashed, black, thin]
        (0,0) -- (4.2,4.2)
        node[right, scale=\diagScale] {$\theta=v$};

      \draw[
        domain=-0.2:4.0,
        smooth,
        variable=\v,
        green!50!black,
        thick
      ]
        plot ({\v}, {\a + exp(\b*(\v - \c))})
        node[right, scale=\diagScale] {$\theta=f(v)$};

      \draw[dashed line]
        (0,\yZero) -- (0,0)
        node[below, black, scale=\lblScale] {$0$};

      \draw[dashed line]
        (\vRest,\yR) -- (\vRest,0)
        node[
          below,
          black,
          scale=\lblScale,
          yshift=-2pt
        ] {$v_{\mathrm{rest}}$};

      \draw[dashed line]
        (\vRA,\yRA) -- (\vRA,0)
        node[below, black, scale=\lblScale]
        {$v_{\mathrm{rest}}+A$};

      \node[left, scale=\lblScale]
        at (0,\vRA)
        {$v_{\mathrm{rest}}+A$};

      \draw[dashed line]
        (\vRest,\yR) -- (0,\yR)
        node[left, scale=\lblScale, black]
        {$f(v_{\mathrm{rest}})$};

      \draw[dashed line]
        (\vRA,\yRA) -- (0,\yRA)
        node[left, scale=\lblScale, black]
        {$f(v_{\mathrm{rest}}+A)$};

      \node[dot] at (\vRest,\yR) {};
      \node[dot] at (\vRA,\yRA) {};
    \end{tikzpicture}
  \end{minipage}%
  \hspace{0.1cm}
  \begin{minipage}{0.3\textwidth}
    \centering
    \scriptsize
    \textbf{\hspace{8pt} Case 2 and 3:
      $f(0) \le v_{\mathrm{rest}} + A
      \le f(v_{\mathrm{rest}}+A)$} \\

    \begin{tikzpicture}[
      scale=\commonScale,
      dashed line/.style={dashed, thin, gray},
      dot/.style={circle, fill=black, inner sep=1pt}
    ]
      \pgfmathsetmacro{\a}{1.5}
      \pgfmathsetmacro{\b}{0.4}
      \pgfmathsetmacro{\c}{0.5}

      \def\vRA{3.5}
      \def\hAddCaseTwo{0.8}

      \pgfmathsetmacro{\yZero}{\a + exp(\b*(0 - \c))}
      \pgfmathsetmacro{\yRA}{\a + exp(\b*(\vRA - \c))}

      \begin{scope}[on background layer]
        \fill[blue!5]
          (0,0) -- (0,6.5) -- (5,6.5) -- (5,5) -- cycle;
      \end{scope}

      \begin{scope}
        \coordinate (Y1) at (0,\yZero);
        \coordinate (Y2) at (\yZero,\yZero);
        \coordinate (Y3) at (\vRA,\vRA);
        \coordinate (Y4) at (\vRA,\yRA);
        \coordinate (Y5) at (0,\yRA);
        \coordinate (Y4prime) at (\vRA,\yRA + \hAddCaseTwo);
        \coordinate (Y5prime) at (0,\yRA + \hAddCaseTwo);

        \fill[orange!35]
          (Y1) -- (Y2) -- (Y3) -- (Y4) -- (Y5) -- cycle;

        \draw[brown!60!black, ultra thick]
          (Y1) -- (Y2) -- (Y3) -- (Y4) -- (Y5) -- cycle;

        \filldraw[fill=orange!10, dotted, thick]
          (Y5) rectangle (Y4prime);

        \draw[dashed, thin, black!70]
          (0,\vRA) -- (\vRA,\vRA);

        \node[below right, scale=0.8]
          at (Y1) {$Y_1$};

        \node[below right, scale=0.8]
          at (Y2) {$Y_2$};

        \node[below right, scale=0.8]
          at (Y3) {$Y_3$};

        \node[above right, scale=0.8]
          at (Y4) {$Y_4$};

        \node[above right, scale=0.8]
          at (Y5) {$Y_5$};

        \node[above right, scale=0.8]
          at (Y4prime) {$Y_4'$};

        \node[above right, scale=0.8]
          at (Y5prime) {$Y_5'$};
      \end{scope}

      \draw[->, >=Stealth]
        (-0.5,0) -- (5,0)
        node[right] {\scriptsize $v$};

      \draw[->, >=Stealth]
        (0,-0.5) -- (0,6.5)
        node[above] {\scriptsize $\theta$};

      \draw[dashed, black, thin]
        (0,0) -- (4.2,4.2)
        node[right, scale=\diagScale] {$\theta=v$};

      \draw[
        domain=-0.2:4.2,
        smooth,
        variable=\v,
        green!50!black,
        thick
      ]
        plot ({\v}, {\a + exp(\b*(\v - \c))})
        node[right, scale=\diagScale] {$\theta=f(v)$};

      \draw[dashed line]
        (0,\yZero) -- (0,0)
        node[below, black, scale=\lblScale] {$0$};

      \draw[dashed line]
        (\vRA,\yRA) -- (\vRA,0)
        node[below, black, scale=\lblScale]
        {$v_{\mathrm{rest}}+A$};

      \node[left, scale=\lblScale]
        at (0,\yZero)
        {$f(0)$};

      \node[left, scale=\lblScale]
        at (0,\vRA)
        {$v_{\mathrm{rest}}+A$};

      \node[left, scale=\lblScale]
        at (0,\yRA)
        {$f(v_{\mathrm{rest}}+A)$};

      \node[left, scale=\lblScale]
        at (0,\yRA + \hAddCaseTwo)
        {$f(v_{\mathrm{rest}}+A)+k$};

      \node[dot] at (0,\yZero) {};
      \node[dot] at (\vRA,\yRA) {};
    \end{tikzpicture}
  \end{minipage}%
  \hspace{0.5cm}
  \begin{minipage}{0.3\textwidth}
    \centering
    \scriptsize
    \textbf{\hspace{7pt} Case 4:
      $v_{\mathrm{rest}} + A
      > f(v_{\mathrm{rest}} + A)$} \\

    \begin{tikzpicture}[
      scale=\commonScale,
      dashed line/.style={dashed, thin, gray},
      dot/.style={circle, fill=black, inner sep=1pt}
    ]
      \pgfmathsetmacro{\a}{0.8}
      \pgfmathsetmacro{\b}{0.25}
      \pgfmathsetmacro{\c}{0.5}

      \def\vRA{4.2}

      \pgfmathsetmacro{\yZero}{\a + exp(\b*(0 - \c))}
      \pgfmathsetmacro{\yRA}{\a + exp(\b*(\vRA - \c))}
      \pgfmathsetmacro{\yTop}{\vRA + \hAdd}

      \begin{scope}[on background layer]
        \fill[blue!5]
          (0,0) -- (0,6.5) -- (5.5,6.5) -- (5.5,5.5) -- cycle;
      \end{scope}

      \begin{scope}
        \coordinate (A) at (0,\yZero);
        \coordinate (B) at (\yZero,\yZero);
        \coordinate (C) at (\vRA,\vRA);
        \coordinate (D) at (\vRA,\yTop);
        \coordinate (E) at (0,\yTop);

        \fill[orange!35]
          (A) -- (B) -- (C) -- (D) -- (E) -- cycle;

        \draw[brown!60!black, ultra thick]
          (A) -- (B) -- (C) -- (D) -- (E) -- cycle;

        \draw[dashed, thin, black!70]
          (0,\vRA) -- (\vRA,\vRA);

        \node[below right, scale=0.8]
          at (A) {$Z_1$};

        \node[below right, scale=0.8]
          at (B) {$Z_2$};

        \node[below right, scale=0.8]
          at (C) {$Z_3$};

        \node[above right, scale=0.8]
          at (D) {$Z_4$};

        \node[above right, scale=0.8]
          at (E) {$Z_5$};
      \end{scope}

      \draw[->, >=Stealth]
        (-0.5,0) -- (5.5,0)
        node[right] {\scriptsize $v$};

      \draw[->, >=Stealth]
        (0,-0.5) -- (0,6.5)
        node[above] {\scriptsize $\theta$};

      \draw[dashed, black, thin]
        (0,0) -- (5.0,5.0)
        node[
          right,
          scale=\diagScale,
          xshift=-7pt,
          yshift=-7pt
        ] {$\theta=v$};

      \draw[
        domain=-0.2:5.0,
        smooth,
        variable=\v,
        green!50!black,
        thick
      ]
        plot ({\v}, {\a + exp(\b*(\v - \c))})
        node[
          below right,
          scale=\diagScale,
          xshift=-14pt,
          yshift=-7pt
        ] {$\theta=f(v)$};

      \draw[dashed line]
        (0,\yZero) -- (0,0)
        node[below, black, scale=\lblScale] {$0$};

      \draw[dashed line]
        (\vRA,\vRA) -- (\vRA,0)
        node[below, black, scale=\lblScale]
        {$v_{\mathrm{rest}}+A$};

      \node[left, scale=\lblScale]
        at (0,\yZero)
        {$f(0)$};

      \node[left, scale=\lblScale]
        at (0,\vRA)
        {$v_{\mathrm{rest}}+A$};

      \node[left, scale=\lblScale]
        at (0,\yTop)
        {$v_{\mathrm{rest}}+A+\Delta$};

      \draw[dashed line]
        (\vRA,\yRA) -- (0,\yRA)
        node[left, scale=\lblScale, black]
        {$f(v_{\mathrm{rest}}+A)$};

      \node[dot] at (0,\yZero) {};
      \node[dot] at (\vRA,\yRA) {};
    \end{tikzpicture}
  \end{minipage}

  \vspace{0.1cm}

  \caption{Four types of PIGARs (in orange) in the MHR system
    depending on the system parameters.
    The four cases are mutually exclusive and complementary.
    Cases 2 and 3 differ in the sign of the number
    $k=\Delta-f(v_{\mathrm{rest}}+A)+v_{\mathrm{rest}}+A$
    (positive or nonpositive).}
  \label{fig:pigars}
\end{figure}
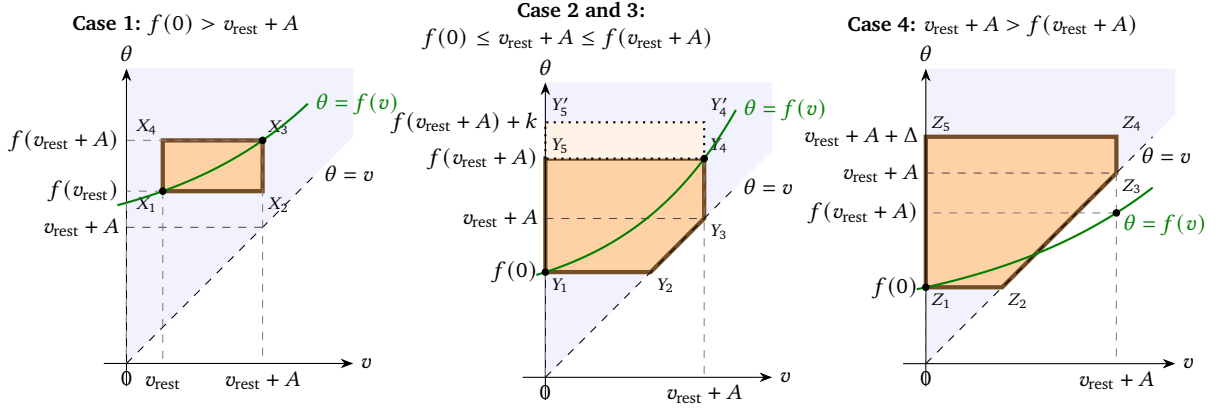
However, which type of PIGAR actually 
appears in the MHR system depends on two factors: 
\begin{itemize}
    \item the position of rectangle $P_\mathrm{basic}$ 
    (introduced in Section~\ref{sec:dynnh}) relative to the diagonal 
    $\theta=v$ (whether the diagonal intersects it 
    and exactly how),
    \item the value $\Delta$ of the jump in the reset condition; 
    more precisely, on the value of the number 
    $$k=\Delta-f(v_{\mathrm{rest}}+A)+v_{\mathrm{rest}}+A$$ 
    (whether positive or non-positive).
\end{itemize}
It turns out that in our analysis four different cases 
naturally arise corresponding to four types
of PIGARs for the MHR system 
(see Table~\ref{tab:cases}).
\begin{table}[h]
    \centering
    \caption{Cases corresponding to four types of PIGARs.} 
    \vspace{1mm}
    \label{tab:cases}
\begin{tabular}{l | c  }
Case 1 & $f(0) > v_{\mathrm{rest}} + A$\\\hline
Case 2 & $f(0) \le v_{\mathrm{rest}} + A 
            \le f(v_{\mathrm{rest}}+A)$ and $k\le0$\\\hline
Case 3 & $f(0) \le v_{\mathrm{rest}} + A 
            \le f(v_{\mathrm{rest}}+A)$ and $k>0$\\\hline
Case 4 & $v_{\mathrm{rest}} + A > f(v_{\mathrm{rest}} + A)$
\end{tabular}
\end{table}

\noindent Note that these four cases are mutually 
exclusive and complementary. 
In light of the four cases mentioned in Table~\ref{tab:cases}, 
let us consider four types of polygons:
$Q_1$, $Q_2$, $Q_3$ and $Q_4$, 
where the polygon $Q_i$ corresponds to Case~$i$ 
for $i=1,2,3,4$. Each type of polygon is
uniquely determined by its set of vertices with
coordinates listed in Table~\ref{tab:points}.
Namely,
$Q_1$: $X_1,X_2,X_3,X_4$ (rectangle);
$Q_2$: $Y_1,Y_2,Y_3,Y_4,Y_5$ (pentagon);
$Q_3$: $Y_1,Y_2,Y_3,Y_4',Y_5'$ (pentagon);
$Q_4$: $Z_1,Z_2,Z_3,Z_4,Z_5$ (pentagon).
Finally, for simplicity of notation, set
$Q=Q_i$ if Case $i$ holds.
Now we are ready to formulate our main result
concerning PIGARs in the MHR system.

\begin{table}[h]
    \centering
    \caption{Coordinates of vertices of four types of 
    polygons.} 
    \vspace{1mm}
    \label{tab:points}
    \resizebox{\textwidth}{!}{
\begin{tabular}{l | l | l | l | l | l c }
$Q_1$ & 
$X_1=(v_{\mathrm{rest}} , f(v_{\mathrm{rest}}))$ & 
$X_2=(v_{\mathrm{rest}}+A , f(v_{\mathrm{rest}}))$ & 
$X_3=(v_{\mathrm{rest}}+A,f(v_{\mathrm{rest}}+A))$ & 
$X_4=(v_{\mathrm{rest}},f(v_{\mathrm{rest}}+A))$  & \\\hline
$Q_2$ & $Y_1=(0,f(0))$ & 
$Y_2=(f(0),f(0))$ & 
$Y_3=(v_{\mathrm{rest}}+A,v_{\mathrm{rest}}+A)$ & 
$Y_4=(v_{\mathrm{rest}}+A,f(v_{\mathrm{rest}}+A))$ & 
$Y_5=(0,f(v_{\mathrm{rest}}+A))$ \\\hline
$Q_3$ & $Y_1=(0,f(0))$ & 
$Y_2=(f(0),f(0))$ & 
$Y_3=(v_{\mathrm{rest}}+A,v_{\mathrm{rest}}+A)$ & 
$Y_4'=(v_{\mathrm{rest}}+A,f(v_{\mathrm{rest}}+A)+k)$ & 
$Y_5'=(0,f(v_{\mathrm{rest}}+A)+k)$ \\\hline
$Q_4$ & $Z_1=(0,f(0))$ & 
$Z_2=(f(0),f(0))$ & 
$Z_3=(v_{\mathrm{rest}}+A,v_{\mathrm{rest}}+A)$ & 
$Z_4=(v_{\mathrm{rest}}+A,v_{\mathrm{rest}}+A+\Delta)$ & 
$Z_5=(0,v_{\mathrm{rest}}+A+\Delta)$ \\
\end{tabular}
}
\end{table}

\begin{theorem}[PIGARs in the MHR system]\label{thm:pigars}
The polygon $Q$ forms a PIGAR of the MHR system, i.e.,
the polygon $Q_i$ is a PIGAR if Case $i$ holds ($i=1,2,3,4$).
\end{theorem}

Before we present the proof of Theorem~\ref{thm:pigars}, 
we must show two auxiliary results: Triangle Lemma
and Barrier Lemma, which also seem to be 
interesting in their own right.

\begin{triang}\label{lem:triangle}
Consider the triangle $A_1A_2A_3$ with vertices 
\[
A_1=(0,0), \quad A_2=(0,f(0)), \quad A_3=(f(0),f(0)).
\]
Then every orbit of the MHR system starting 
in this triangle leaves it after a finite
(possibly zero) number of resets
and never returns to this triangle. 
Moreover, the number of resets is bounded by
\[
N \le \left\lfloor \frac{f(0)-\theta_0}{\Delta} \right\rfloor + 1,
\]
where $\theta_0$ is the initial value of $\theta$.
\end{triang}

\begin{proof}
We divide the proof into a few simple steps.\vspace{0.5mm}

\noindent\emph{Step 1: Monotonic growth of $\theta$.}
Inside the triangle $A_1A_2A_3$ we have $\theta \le f(0)$ and $v \ge 0$. 
Since $f$ is increasing, $\theta \le f(0) \le f(v)$ for all 
$v \in [0,f(0)]$.
Hence, in the interior of the triangle,
\(
\frac{d\theta}{dt} = \frac{f(v)-\theta}{\tau} > 0,
\)
so $\theta(t)$ is strictly increasing 
during continuous evolution.\vspace{0.5mm}

\noindent\emph{Step 2: Orbits without reset inside the triangle $A_1A_2A_3$
leave it through $\theta = f(0)$ 
level.}
Assume no reset in the triangle $A_1A_2A_3$ occurs
for some orbit of the MHR system. 
Hence this orbit can be treated
inside the triangle $A_1A_2A_3$
as an orbit of the non-hybrid NH system. 
But, by Main Theorem on the NH system, such an orbit
tends asymptotically to the rectangle $P_\mathrm{basic}$,
whose bottom side lies at the level $\theta=f(v_\mathrm{rest})>f(0)$.
Consequently, since our orbit has no resets inside the triangle
and the non-autonomous vector field $F$ points to the right along 
the $\theta$-axis, the orbit must leave the triangle $A_1A_2A_3$
through the $A_2A_3$ segment.
\vspace{0.5mm}

\noindent\emph{Step 3: Effect of resets.}
A reset occurs when $v=\theta$ and the solution
intersects the diagonal. At such a time,
\[
v(t^+) = 0, \qquad \theta(t^+) = \theta(t^-) + \Delta,
\]
with $\Delta>0$. Thus each reset increases $\theta$ by 
$\Delta$ and places the orbit on the $\theta$-axis.\vspace{0.5mm}

\noindent\emph{Step 4: Uniform bound on the number of resets.}
Since $\theta$ increases between resets (see Step 1),
after $k$ resets,
\[
\theta \ge \theta_0 + k\Delta.
\]
As long as the trajectory remains in the triangle, we must have $\theta \le f(0)$. Therefore,
\[
\theta_0 + k\Delta \le f(0),
\]
which implies
\[
k \le \frac{f(0)-\theta_0}{\Delta}.
\]
Hence the number of resets is bounded by
\[
N \le \left\lfloor \frac{f(0)-\theta_0}{\Delta} \right\rfloor + 1.
\]

\noindent\emph{Step 5: Exit from the triangle.}
By Step 4, as long as the trajectory remains inside the triangle, 
infinitely many resets are impossible. 
Therefore, after a finite number of resets, 
if the trajectory has not yet left the triangle, 
it evolves inside the triangle without further resets. 
Then Step 2 applies, and the trajectory must eventually 
leave the triangle through the upper side $A_2A_3$
at the level $\theta=f(0)$.\vspace{0.5mm}

\noindent\emph{Step 6: No returns to the triangle.}
Note that since the scalar product satisfies
\[
\langle F\mid [0,-1]\rangle=\big(\theta-f(v)\big)/\tau\le0
\]
for all points on the segment $A_2A_3$, no orbit can enter 
the triangle through these points.

This proves that every orbit leaves the triangle
$A_1A_2A_3$ after finitely many resets and never comes back.
\end{proof}

\begin{corollary}
Broadly speaking, Triangle Lemma states 
that the dynamics of our hybrid system in the $A_1A_2A_3$ 
triangle is extremely simple and transient.
Consequently, it can be generally disregarded 
when analyzing the MHR system.   
\end{corollary}

\begin{figure}[htbp]
  \centering
  \begin{tikzpicture}[scale=2]
    \draw[->]
      (0,0) -- (2.2,0)
      node[right] {$v$};

    \draw[->]
      (0,0) -- (0,2.2)
      node[above] {$\theta$};

    \def\fzero{1}

    \coordinate (A) at (0,0);
    \coordinate (B) at (0,\fzero);
    \coordinate (C) at (\fzero,\fzero);

    \fill[gray!20]
      (A) -- (B) -- (C) -- cycle;

    \draw[thick]
      (A) -- (B) -- (C) -- cycle;

    \fill
      (A) circle (0.02)
      node[below left] {$A_1$};

    \fill
      (B) circle (0.02)
      node[left] {$A_2$};

    \fill
      (C) circle (0.02)
      node[right] {$A_3$};

    \draw[dashed]
      (0,0) -- (2,2)
      node[above right] {$v=\theta$};

    \draw[
      green!60!black,
      thick,
      domain=0:1.3
    ]
      plot (\x,{\fzero + 0.25*(exp(1.2*\x)-1)});

    \node[green!60!black]
      at (1.4,2.1)
      {$\theta=f(v)$};
  \end{tikzpicture}

  \caption{The triangle from Triangle Lemma
    with simple transient dynamics.}
\end{figure}
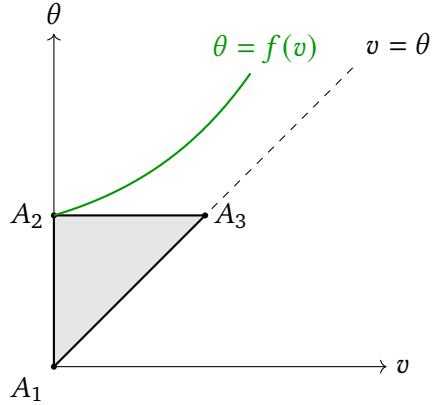

The following result states that both in the NH system
and in the MHR system
there are two natural barriers: vertical at $v=v_{\mathrm{rest}}+A$
for low voltages (lower than $v_{\mathrm{rest}}+A$)
and horizontal at $\theta=f(0)$ for high thresholds
(higher than $f(0)$). See Figure~\ref{fig:barriers}
for details. Precisely the following
estimations hold.

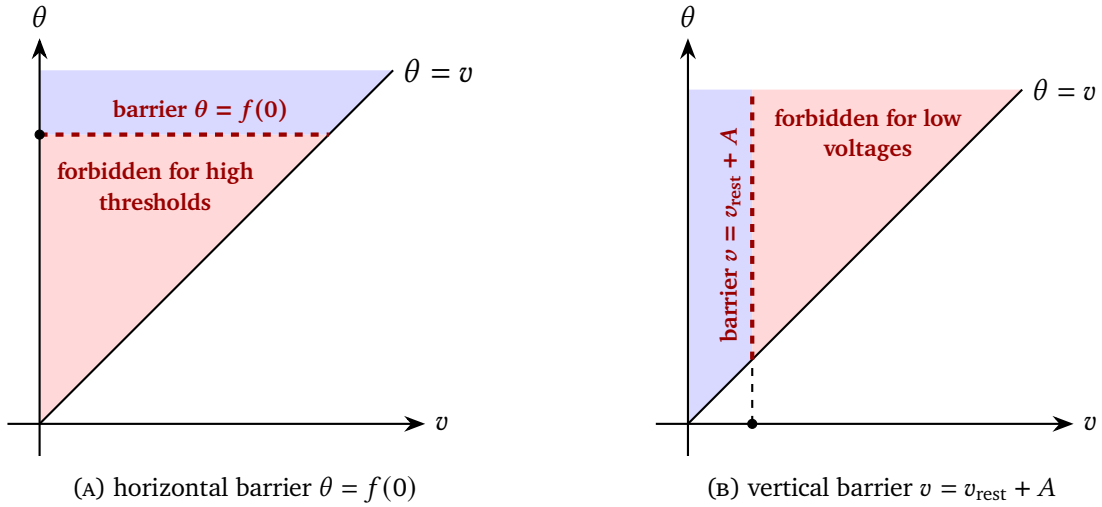
\begin{figure}[htbp]
  \centering

  \tikzset{
    axis/.style={thick, -{Stealth[scale=1.2]}},
    phase_space_above/.style={fill=blue!15, draw=none},
    phase_space_below/.style={fill=red!15, draw=none},
    boundary_line/.style={thick, black}
  }

  \begin{subfigure}[b]{0.48\textwidth}
    \centering
    \begin{tikzpicture}[scale=0.85]
      \fill[phase_space_below]
        (0,0) -- (4.5,4.5) -- (0,4.5) -- cycle;

      \node[
        align=center,
        font=\footnotesize\bfseries,
        text=darkred
      ] at (1.8,3.7) {forbidden for high\\thresholds};

      \fill[phase_space_above]
        (0,4.5) -- (4.5,4.5) -- (5.5,5.5) -- (0,5.5) -- cycle;

      \draw[axis]
        (-0.5,0) -- (6.0,0)
        node[right] {$v$};

      \draw[axis]
        (0,-0.5) -- (0,6.0)
        node[above] {$\theta$};

      \draw[boundary_line]
        (0,0) -- (5.5,5.5)
        node[right, black, font=\large\bfseries] {$\theta = v$};

      \draw[ultra thick, dashed, darkred]
        (0,4.5) -- (4.5,4.5)
        node[
          above left,
          pos=0.9,
          font=\footnotesize\bfseries,
          text=darkred
        ] {barrier $\boldsymbol{\theta = f(0)}$};

      \filldraw[black] (0,4.5) circle (2pt);
    \end{tikzpicture}
    \caption{horizontal barrier $\theta = f(0)$}
  \end{subfigure}
  \hfill
  \begin{subfigure}[b]{0.48\textwidth}
    \centering
    \begin{tikzpicture}[scale=0.85]
      \fill[phase_space_above]
        (0,0) -- (1.0,1.0) -- (1.0,5.2) -- (0,5.2) -- cycle;

      \fill[phase_space_below]
        (1.0,1.0) -- (5.2,5.2) -- (1.0,5.2) -- cycle;

      \draw[axis]
        (-0.5,0) -- (6.0,0)
        node[right] {$v$};

      \draw[axis]
        (0,-0.5) -- (0,6.0)
        node[above] {$\theta$};

      \draw[boundary_line]
        (0,0) -- (5.2,5.2)
        node[right, black] {$\theta = v$};

      \draw[thick, black, dashed]
        (1.0,0) -- (1.0,1.0);

      \draw[ultra thick, dashed, darkred]
        (1.0,1.0) -- (1.0,5.2)
        node[
          rotate=90,
          above left,
          pos=0.9,
          font=\footnotesize\bfseries,
          text=darkred
        ] {barrier $\boldsymbol{v = v_{\mathrm{rest}} + A}$};

      \node[
        align=center,
        font=\footnotesize\bfseries,
        text=darkred
      ] at (2.8,4.5) {forbidden for low\\voltages};

      \filldraw[black] (1.0,0) circle (2pt);
    \end{tikzpicture}
    \caption{vertical barrier
      $v = v_{\mathrm{rest}} + A$}
  \end{subfigure}

  \caption{Two barriers for threshold and
    voltage in the MHR system.}
  \label{fig:barriers}
\end{figure}

\begin{barrier}\label{prop:barrier}
Consider the solution of either the NH system
or the MHR system with the initial
condition $(v_0,\theta_0,t_0)$. Then
\begin{enumerate}
    \item if $v_0\le v_{\mathrm{rest}}+A$ then
    $v(t)\le v_{\mathrm{rest}}+A$,
    \item if $v_0\ge0$ and $\theta_0\ge f(0)$ then
    $\theta(t)\ge f(0)$ (note that the assumption $v_0\ge0$ is 
    always fulfilled in the MHR system).
\end{enumerate}  
\end{barrier}

\begin{proof}
Let us start with the NH system and
recall the general formulas for solutions.

\emph{Ad 1.} Since
\[
v(t)=v_\mathrm{rest}+
(v_0-v_\mathrm{rest})\ue^{t_0-t}+\ue^{-t}\int_{t_0}^t\!\!\! I(s)\ue^{s}ds,
\]
if $v_0\le v_{\mathrm{rest}}+A$ then
\[
v(t)\le v_\mathrm{rest}+A\ue^{t_0-t}
+A\ue^{-t}[\ue^{t}-\ue^{t_0}]=v_\mathrm{rest}+A.
\]

\emph{Ad 2.} Note that if $v_0\ge0$ then $v(t)\ge0$ for $t\ge t_0$
and, in consequence, $f(v(t))\ge f(0)$ for $t\ge t_0$.
Now, since 
\[
\theta(t)=
\ue^{(t_0-t)/\tau}\theta_0+
\ue^{-t/\tau}\frac1\tau
\int_{t_0}^t\!\! f\big(v(s)\big)\ue^{s/\tau}\,ds,
\]
if $v_0\ge0$ and $\theta_0\ge f(0)$ then
\[
\theta(t)\ge
\ue^{(t_0-t)/\tau}\theta_0+
f(0)\ue^{-t/\tau}[\ue^{t/\tau}-\ue^{t_0/\tau}]=
\ue^{(t_0-t)/\tau}[\theta_0-f(0)]+f(0)\ge f(0).
\]
However, the above estimations work also for 
the hybrid MHR system. Namely, if the condition
$v_0\le v_{\mathrm{rest}}+A$ (resp. $v_0\ge0$ and $\theta_0\ge f(0)$)
was satisfied before the reset then it will be also 
satisfied after the reset.
\end{proof}

\begin{proof}[Proof of Theorem~\ref{thm:pigars}]
We divide the proof into two parts
and each part into two steps.\vspace{1mm}

\noindent\textsc{Part I: positive invariance.} Here we prove 
that the polygon $Q$ is positively invariant.\vspace{1mm}

\noindent\textit{Step 1.}
Firstly, we analyze the behavior of the non-autonomous 
vector field $F$ on the boundary of the polygon $Q$,
where there is no reset phenomenon—that is, 
excluding the segment on the diagonal $\theta=v$, 
which we will examine separately in the second step.
It is easy to see that in each of the four cases, 
the considered part of the perimeter of the polygon $Q$
(i.e., the part without the diagonal) is contained in
the perimeter of some admissible rectangle $P$
from Proposition~\ref{prop:naode}, which guarantees that no orbit 
can leave $Q$ through the points on this part of the perimeter.
Namely, one can check that we should take 
in Case 1: $P=P_\mathrm{basic}$,
in Case 2: $P=P_{v_{\mathrm{rest}},0}^{0,0}$,
in Case 3: $P=P_{v_{\mathrm{rest}},0}^{0,k}$
and in Case 4: $P=P_{v_{\mathrm{rest}},0}^{0,m}$
with $m=v_{\mathrm{rest}}+A+\Delta-f(v_{\mathrm{rest}}+A)$\vspace{1mm}.
    
\noindent\textit{Step 2.} Secondly, let us analyze what happens 
on the segment of the polygon $Q$'s perimeter that lies 
on the diagonal $\theta=v$.
Since in Case 1 the polygon $Q$ is disjoint from the diagonal, 
we need to consider only the remaining three cases. 
In each of these cases, if a reset occurs then the point
from the intersection of $Q$ and the diagonal
(note that for such a  point we have $f(0)\le v\le v_{\mathrm{rest}}+A$)
is transferred from the diagonal to the $\theta$-axis 
while increasing the $\theta$ value by $\Delta$,
and is therefore transferred up to a maximum height of 
$\theta^\#=v_{\mathrm{rest}}+A+\Delta$ on the $\theta$-axis.
However, by the definition of $Q$, in each of the three cases, 
$Q$ contains a segment of the $\theta$-axis up to the height 
$\theta^*$ such that $\theta^*\ge\theta^\#$. Namely, in Case 2:
$\theta^*=f(v_{\mathrm{rest}}+A)=v_{\mathrm{rest}}+A+\Delta-k
\ge v_{\mathrm{rest}}+A+\Delta$
(since $k\le0$), in Case 3: 
$\theta^*=f(v_{\mathrm{rest}}+A)+k=v_{\mathrm{rest}}+A+\Delta$ 
(here $k>0$) and in Case 4: 
$\theta^*=v_{\mathrm{rest}}+A+\Delta$.\vspace{2mm}

\noindent\textsc{Part II: global attractor.} Now we prove
the following condition $(*)$: for every arbitrarily small 
$\gamma>0$ and every initial condition
$(v_0,\theta_0,t_0)$ from the extended phase space
there is $\overline{t}\ge t_0$ such that 
for all $t\ge\overline{t}$ we have
\(
\dist\big((v(t),\theta(t)),Q\big)<\gamma.
\)\vspace{1mm}

\noindent\textit{Step 1.}
First consider orbits with finite number of resets.
By Proposition 4.1,
such orbits converge to the attracting periodic orbit 
$\mathcal{O} \subset P_{\mathrm{basic}} \cap \mathbb{P} \subset Q$. 
Hence, condition $(*)$ is satisfied in this case.\vspace{1mm}

\noindent\textit{Step 2.}
Next consider orbits with infinite number of resets. 
Their behavior depends on the initial condition. 
To carry out a detailed analysis, we will need some 
additional notation. Let us define the numbers
\[
\tilde{v}:=v_{\mathrm{rest}}+A,
\qquad
\tilde{\theta}:=\max\{\tilde{v},f(\tilde{v})\}.
\]
and the sets
\[
W=\{(v,\theta)\mid 0\le v\le\tilde{v},\ f(0)\le\theta\le\tilde{\theta}\},
\qquad
W'=W\cap\mathbb P.
\]
Moreover, outside $W'$ we distinguish three regions
in the phase space $\mathbb P$:
\begin{align*}
\Gamma_1&=\{(v,\theta)\in\mathbb P:
0\le v\le\tilde{v},\ \theta>\tilde{\theta}\},\\[1mm]
\Gamma_2&=\{(v,\theta)\in\mathbb P:
0\le v<\tilde{v},\ \theta<f(0)\},\\[1mm]
\Gamma_3&=\{(v,\theta)\in\mathbb P:
\theta\ge v>\tilde{v}\}.
\end{align*}
The locations of all the above sets are shown in 
Figure~\ref{fig:regions} in two different situations: 
$f(0)>\tilde{v}$ and $f(0)\le\tilde{v}$.
However, in both situations, 
the proof proceeds in exactly the same way.

\begin{figure}[htbp]
  \centering
  \begin{tikzpicture}[scale=0.85]

    \begin{scope}[shift={(9,0)}]

      \def\V{4}
      \def\F{1.8}
      \def\T{5.4}

      \draw[->,thick]
        (0,0)--(7,0)
        node[right] {$v$};

      \draw[->,thick]
        (0,0)--(0,7)
        node[above] {$\theta$};

      \draw[dashed,thick]
        (0,0)--(6.5,6.5)
        node[above right] {$\theta=v$};

      \fill[blue!20]
        (0,0)--(\F,\F)--(0,\F)--cycle;

      \fill[green!25]
        (0,\F)--(0,\T)--(\V,\T)--(\V,\V)--(\F,\F)--cycle;

      \fill[red!20]
        (0,\T)--(0,6.5)--(\V,6.5)--(\V,\T)--cycle;

      \fill[orange!25]
        (\V,\V)--(\V,6.5)--(6.5,6.5)--cycle;

      \draw[dotted]
        (\V,0)--(\V,6.5);

      \draw[dotted]
        (0,\F)--(\V,\F);

      \draw[dotted]
        (0,\T)--(\V,\T);

      \draw[dotted]
        (0,\T+0.3)--(\V,\T+0.3);

      \node at (2.2,3.3) {$W'$};
      \node at (2.0,6.0) {$\Gamma_1$};
      \node at (5.2,5.8) {$\Gamma_3$};
      \node at (0.75,1.3) {$\Gamma_2$};

      \node[left] at (0,\F) {$f(0)$};
      \node[left] at (0,\T) {$\tilde{\theta}$};
      \node[below] at (\V,0) {$\tilde{v}$};

      \node at (3.2,-0.8) {situation $f(0)\le\tilde{v}$};

    \end{scope}


    \begin{scope}[shift={(0,0)}]

      \def\V{4}
      \def\F{4.8}
      \def\T{5.8}

      \draw[->,thick]
        (0,0)--(7,0)
        node[right] {$v$};

      \draw[->,thick]
        (0,0)--(0,7)
        node[above] {$\theta$};

      \draw[dashed,thick]
        (0,0)--(6.5,6.5)
        node[above right] {$\theta=v$};

      \fill[blue!20]
        (0,0)--(\V,\V)--(\V,\F)--(0,\F)--cycle;

      \fill[green!25]
        (0,\F)--(0,\T)--(\V,\T)--(\V,\F)--cycle;

      \fill[red!20]
        (0,\T)--(0,6.5)--(\V,6.5)--(\V,\T)--cycle;

      \fill[orange!25]
        (\V,\V)--(\V,6.5)--(6.5,6.5)--cycle;

      \draw[dotted]
        (\V,0)--(\V,6.5);

      \draw[dotted]
        (0,\F)--(\V,\F);

      \draw[dotted]
        (0,\T)--(\V,\T);

      \node at (2.0,5.3) {$W'$};
      \node at (2.0,6.15) {$\Gamma_1$};
      \node at (5.2,5.8) {$\Gamma_3$};
      \node at (1.8,3) {$\Gamma_2$};

      \node[left] at (0,\F) {$f(0)$};
      \node[left] at (0,\T) {$\tilde{\theta}$};
      \node[below] at (\V,0) {$\tilde{v}$};

      \node at (3.2,-0.8) {situation $f(0)>\tilde{v}$};

    \end{scope}
  \end{tikzpicture}

  \caption{Four regions from the proof of Theorem~\ref{thm:pigars}.
    Recall that $\tilde{v}:=v_{\mathrm{rest}}+A$ and
    $\tilde{\theta}:=\max\{\tilde{v},f(\tilde{v})\}$.}
  \label{fig:regions}
\end{figure}
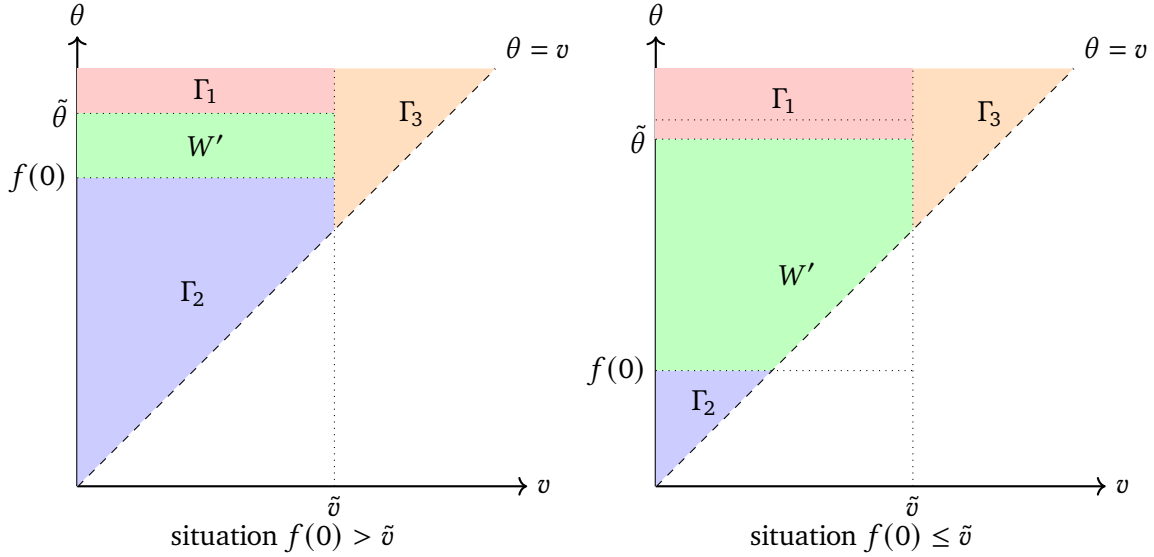
To finish the proof, we will examine four scenarios depending on 
the starting point $(v_0,\theta_0)$. Recall that now we 
consider only orbits with infinite number of resets.\vspace{1mm}

\noindent\textbf{A.} Suppose the orbit starts in $W'$. Then:
\begin{itemize}
    \item In Case 1, this leads to a contradiction. 
    Indeed, in this case $W' = W$ is 
    positively invariant and disjoint from the diagonal.
    Hence there are no resets in $W'$.
    \item In Cases 2, 3, and 4, we have $W' \subset Q$. 
    Since $Q$ is positively invariant, an orbit starting 
    in $W'$ does not leave $Q$. Thus, condition $(*)$ 
    is satisfied.
\end{itemize}

\noindent\textbf{B.} Suppose the orbit starts in $\Gamma_1$. 
It cannot cross $v=\tilde{v}$ (the Barrier Lemma)
and, by assumption, has resets. Hence it must enter $W'$
to have these resets, i.e., to reach the diagonal. 
Thus the situation reduces to case \textbf{A}.\vspace{1mm}

\noindent\textbf{C.} Suppose the orbit starts in $\Gamma_2$. 
By the Triangle Lemma, it must enter $W'$ or $\Gamma_1$
(in the latter case, if there is a sufficiently 
large jump after a reset),
because it must cross or jump over the line $\theta = f(0)$ 
and cannot cross $v=\tilde{v}$ (by the Barrier Lemma). 
Now the situation reduces 
to case \textbf{A} or \textbf{B}.\vspace{1mm}

\noindent\textbf{D.} Finally, suppose the orbit starts in $\Gamma_3$. Then either:
\begin{itemize}
    \item it undergoes a reset in $\Gamma_3$ and, as a result,
    moves to the part of the $\theta$-axis 
    in $\Gamma_1$, $\Gamma_2$ or $W'$;
    however, since the orbit undergoes an infinite number of resets,
    we find ourselves in situation \textbf{A}, \textbf{B} or
    \textbf{C} and we are done,~ or
    \item it does not undergo a reset in $\Gamma_3$, 
    in which case it must enter $W'$ (either directly or via $\Gamma_1$), because, by assumption, it has an infinite
    number of resets; once more we are in case \textbf{A}.
\end{itemize}
This completes the proof.
\end{proof}

\begin{remark}\label{rem:regions}
Note that the regions introduced in the proof of 
Theorem~\ref{thm:pigars} have the following obvious 
but interesting properties.
\begin{itemize}
    \item As long as the orbit is in $\Gamma_1$, 
    there are no resets ($\Gamma_1$ is disjoint
    from the diagonal).
    \item An orbit that starts in $\Gamma_2$, after 
    a finite (perhaps zero) number of resets, 
    leaves $\Gamma_2$ and never returns there
    (see Triangle Lemma). 
    \item In $\Gamma_3$, there is at most one reset, after which the orbit is transferred to one 
    of the other regions.
    \item For $f(0) > v_{\mathrm{rest}} + A$
    there are no resets in $W'$. Hence if 
    $f(0) > v_{\mathrm{rest}} + A$ then
    each orbit of the MHR system has only finite 
    (possibly zero) number of resets and converges
    to the reset-free periodic orbit.
    In other words, this case (Case 1) is not 
    fundamentally different from a non-hybrid 
    situation in the sense that all resets 
    in the hybrid system are transient.
\end{itemize}
In consequence, an orbit with an infinite number 
of resets can only occur if 
$f(0)\le v_{\mathrm{rest}} + A$.
Moreover, every such orbit must enter $W'$
and remain in $Q$ forever (note that $W'$ need not 
be positively invariant).
\end{remark}

\subsection{Finite number of resets 
during one period}
In this subsection we provide some estimation
for the number of possible spikes 
during one period in the MHR system.
In particular, we prove that this number is not only finite,
but also uniformly bounded from above 
for all initial conditions.
Let us start with three auxiliary technical results.

\begin{lemma}[Mean value inequality for piecewise smooth functions]
\label{lem:mean}
Let
\(
f:[a,b]\to \mathbb{R}
\)
be continuous on $[a,b]$. Assume that there exists a partition
\[
a=x_0<x_1<\dots<x_n=b
\]
such that
\(
f\in C^1((x_{k-1},x_k))
\)
for every $k=1,\dots,n$, and that the one-sided 
derivatives exist at each partition point.
Suppose furthermore that there exist constants 
$m,M\in\mathbb{R}$ such that
\[
m\le f'(x)\le M
\]
for all points where $f'$ exists, including 
the one-sided derivatives at the partition points.
Then
\[
m(b-a)\le f(b)-f(a)\le M(b-a).
\]
\end{lemma}

\begin{proof}
Apply the classical Lagrange mean value theorem on each interval
\(
(x_{k-1},x_k),
\)
obtaining
\[
m(x_k-x_{k-1})
\le
f(x_k)-f(x_{k-1})
\le
M(x_k-x_{k-1}).
\]
Summing over $k=1,\dots,n$ yields
\[
m(b-a)\le f(b)-f(a)\le M(b-a). \qedhere
\]
\end{proof}

\begin{lemma}\label{lem:compact}
Let $G \subset \mathbb{P}$ be a compact set disjoint 
from the diagonal
$
\mathbb{L}=\{(v,\theta)\in\mathbb{P}\mid v=\theta\},
$
and let $q\colon G\to [0,\infty)$ be defined by
$
q(v,\theta)=\theta-v.
$
Then
$
D:=\min\{q(v,\theta)\mid(v,\theta)\in G\}>0.
$
\end{lemma}

\begin{proof}
Since $q$ is continuous and $G$ is compact, 
$q$ attains its minimum on $G$. Suppose, 
contrary to the claim, that $D=0$. 
Then there exists $(v,\theta)\in G$ such that
$
q(v,\theta)=0.
$
Hence $\theta-v=0$, that is, $\theta=v$. 
Therefore $(v,\theta)$ belongs to the diagonal, 
a contradiction.
\end{proof}

\begin{lemma}\label{lem:deriv}
Let $H\subset\mathbb{P}$ be a compact positively invariant set. 
Then there exists a constant $C>0$ such that for every 
$t_0\in\mathbb{R}$, every $t\ge t_0$, and 
every $(v_0,\theta_0)\in H$, 
the corresponding solution satisfies
\[
s(t)=s(t;t_0,v_0,\theta_0)
=\frac{f(v(t))-\theta(t)}{\tau}+v(t)-v_{\mathrm{rest}}-I(t)
>-C.
\]
\end{lemma}

\begin{proof}
Define the function $p\colon H\to\mathbb{R}$ by
$
p(v,\theta)=(f(v)-\theta)/\tau+v.
$
Since $p$ is continuous and $H$ is compact, 
$p$ attains its minimum on $H$. 
Hence there exists a constant $C_1>0$ such that
$
p(v,\theta)>-C_1
$
for all $(v,\theta)\in H$.
Let $(v_0,\theta_0)\in H$. Since $H$ is positively invariant,
$
(v(t),\theta(t))\in H
$
for all $t\ge t_0$.
Therefore,
\[
\frac{f(v(t))-\theta(t)}{\tau}+v(t)>-C_1.
\]
Consequently, using the bound $I(t)\le A$, we obtain
\[
s(t)
=\frac{f(v(t))-\theta(t)}{\tau}+v(t)-v_{\mathrm{rest}}-I(t)
>-C_1-v_{\mathrm{rest}}-A=-C. \qedhere
\]
\end{proof}

\begin{lemma}\label{lem:dc}
Let $H\subset\mathbb{P}$ be a compact positively invariant set, and let
$G\subset H$ be a compact set disjoint from the diagonal. Then there exist
constants $C,D>0$ such that for every reset-including solution with initial
condition $(t_0,v_0,\theta_0)$, where $(v_0,\theta_0)\in G$, if $t_1$ denotes
the first reset time after $t_0$, then
\[
t_1-t_0\ge \frac{D}{C}.
\]
\end{lemma}

\begin{proof}
Consider the ``distance'' function to the reset set (the diagonal)
\[
d(t)=d(t;t_0,v_0,\theta_0):=\theta(t)-v(t).
\]
On the interval $[t_0,t_1^-]$, that is, before the first reset occurs, 
the solution and, in consequence, the function $d$ is piecewise 
smooth with one-sided derivatives at the partition points. 
Moreover, by Lemma~\ref{lem:deriv},
\[
\dot d(t) = s(t)>-C
\]
for all $t\in[t_0,t_1^-]$,
which, when combined with Lemma~\ref{lem:mean}, gives
\[
d(t)-d(t_0)\ge -C(t-t_0).
\]
Since $(v_0,\theta_0)\in G$, Lemma~\ref{lem:compact} yields
\[
d(t_0)=\theta_0-v_0=q(v_0,\theta_0)\ge D>0.
\]
Hence
\[
d(t)\ge d(t_0)-C(t-t_0)
    \ge D-C(t-t_0).
\]
Finally, taking $t=t_1^-$ we obtain
\[
0=d(t_1^-)\ge D-C(t_1-t_0),
\quad\text{i.e.,}\quad
t_1-t_0\ge D/C.\qedhere
\]
\end{proof}

Below, we provide estimates of the time interval between 
any two consecutive resets for all reset-including
solutions of the MHR system. Let us call the time interval 
between successive resets the \emph{interspike interval}.

\begin{figure}[htbp]
  \centering

  \definecolor{skyblue}{RGB}{205,235,255}
  \definecolor{pastelgreen}{RGB}{145,210,145}
  \definecolor{darkgreen}{RGB}{20,70,20}

  \begin{tikzpicture}[scale=1.2, >=Stealth, font=\large]
    \def\vzero{3.0}
    \def\thetazero{4.5}
    \def\deltazero{1.5}
    \def\thetamax{5.0}

    \begin{scope}[on background layer]
      \fill[pastelgreen]
        (0,0) -- (\vzero,\vzero) -- (\vzero,\thetazero) -- (0,\thetazero) -- cycle;

      \fill[skyblue]
        (0,\thetazero) -- (\thetazero,\thetazero) -- (\thetamax,\thetamax) -- (0,\thetamax) -- cycle;

      \fill[skyblue]
        (\vzero,\vzero) -- (\vzero,\thetazero) -- (\thetazero,\thetazero) -- cycle;
    \end{scope}

    \draw[->, ultra thick]
      (-1.0,0) -- (5.5,0)
      node[right, font=\Large] {$v$};

    \draw[->, ultra thick]
      (0,-1.0) -- (0,5.5)
      node[above, font=\Large] {$\theta$};

    \node[below left, font=\large]
      at (0,0) {$O$};

    \node[darkgreen, font=\Huge\bfseries]
      at (1.1,2.9) {$H$};

    \draw[thick, black, domain=0:5.0]
      plot (\x,{\x})
      node[right, black] {$\theta = v$};

    \draw[thick, blue, dashed]
      (\vzero,0) -- (\vzero,5.5)
      node[above, blue] {$v = v^\circ$};

    \draw[thick, orange!80!black, dashed]
      (0,\thetazero) -- (5.5,\thetazero)
      node[right, orange!80!black] {$\theta = \theta^\circ$};

    \draw[line width=2.5pt, red]
      (0,\thetazero) -- (\vzero,\thetazero)
      node[midway, above, red, font=\Large\bfseries] {$G^\ast$};

    \draw[line width=2.5pt, red]
      (0,\deltazero) -- (0,\thetazero)
      node[pos=0.75, left=2pt, red, font=\Large\bfseries] {$G$};

    \draw[
      line width=1.5pt,
      green!50!black,
      shorten <=1.5pt,
      shorten >=1.5pt
    ]
      (0,\deltazero) -- (0,0) -- (\vzero,\vzero) -- (\vzero,\thetazero);

    \draw[thick]
      (\vzero,0.1) -- (\vzero,-0.1)
      node[below, font=\large] {$v^\circ$};

    \draw[thick]
      (0.1,\thetazero) -- (-0.1,\thetazero)
      node[left, font=\large] {$\theta^\circ$};

    \draw[thick]
      (0.1,\deltazero) -- (-0.1,\deltazero)
      node[left, font=\large] {$\Delta$};

    \draw[gray, dotted, thick]
      (\vzero,\vzero) -- (0,\vzero);

    \draw[thick]
      (0.1,\vzero) -- (-0.1,\vzero)
      node[left, font=\large] {$v^\circ$};

    \draw[gray, dotted, thick]
      (\thetazero,\thetazero) -- (\thetazero,0);

    \draw[thick]
      (\thetazero,0.1) -- (\thetazero,-0.1)
      node[below, font=\large] {$\theta^\circ$};
  \end{tikzpicture}

  \caption{Sets $H$, $G$ and $G^\ast$ from the proof
    of Proposition~\ref{prop:est}.}
  \label{fig:setH}
\end{figure}
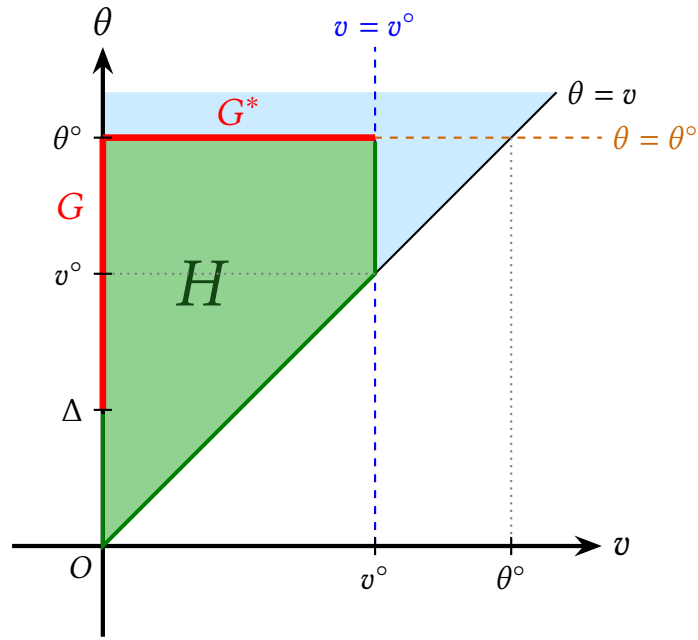

\begin{proposition}\label{prop:est}
There is $\delta>0$ such that for all solutions of the MHR system
all the interspike intervals $\mathbb{t}$ satisfy the~condition
$\mathbb{t}\ge\delta$, i.e., they are uniformly bounded from below
by a positive number.
\end{proposition}

\begin{proof}
Define the numbers
\[
v^\circ:=v_{\mathrm{rest}}+A,
\qquad
\theta^\circ:=\max\{v^\circ+\Delta,f(v^\circ)+\max\{0,k\}\}.
\]
and the set
\[
H=\{(v,\theta)\in\Ph\mid v\le v^\circ,\;\theta\le\theta^\circ\}.
\]
One can check that $H$ is a positively invariant compact trapezoid
(see Figure~\ref{fig:setH}). Positive invariance follows
from an easy analysis similar to that in the proof of 
Proposition~\ref{prop:naode} and Theorem~\ref{thm:pigars}.
Moreover, by Remark~\ref{rem:regions}, each solution 
of the MHR system has at most one reset before entering $H$, 
and all other resets (if they happen to show up)
occur within $H$ (from positive invariance).
We distinguish two cases.\vspace{1mm}

\noindent
\textit{Case 1. Orbits starting in $H$.}
Let $G:=\{(0,\theta)\mid \Delta\le\theta\le\theta^\circ\}$,
i.e., the segment of the reset line contained in $H$
(see Remark~\ref{rem:regions}). 
Observe that after each reset, the orbit starting in $H$
is placed on $G$ and remains in $H$.
Since $G$ is compact, contained in $H$, and disjoint from the diagonal, 
by Lemma~\ref{lem:dc}, there exist constants $C,D>0$
such that for every reset-including solution 
the time elapsed between any two consecutive resets $t_1<t_2$
is bounded from below by $D/C$. Note that at $t_1^+$ 
the orbit is on $G$ and at $t_2^-$ is on the diagonal.
Hence every interspike interval of a solution starting in $H$ 
satisfies $\mathbb t\ge D/C$.\vspace{1mm}

\noindent
\textit{Case 2. Orbits starting outside $H$.}
Recall that an orbit starting outside $H$ can experience 
at most one reset before entering $H$.
Let $G^\ast:=\{(v,\theta^\circ)\mid 0\le v\le v^\circ\}$, i.e., 
the segment of the upper horizontal boundary of $H$
through which reset-including orbits enter $H$. 
The set $G^\ast$ is compact, contained
in $H$, and disjoint from the diagonal. 
Applying again Lemma~\ref{lem:dc} to $G^\ast$, 
we obtain a positive lower bound ${D^\ast}/{C^\ast}$
for the time needed to reach the next reset after entering $H$.
Consequently, if an orbit undergoes a reset before entering $H$, then
the time between this reset and the next one is at least the time elapsed
from entering $H$ to the subsequent reset, and therefore
$
\mathbb t\ge {D^\ast}/{C^\ast}.
$
After that reset, the orbit belongs to Case~1, and all remaining
interspike intervals satisfy the bound obtained there.\vspace{1mm}

Combining the two cases, every interspike interval of every solution is bounded
from below by
\[
\delta:=\min\left\{{D}/{C},\,{D^\ast}/{C^\ast}\right\}>0,
\]
which proves our claim.
\end{proof}

Now we can formulate the  main result of this subsection.

\begin{proposition}
Fix the time interval $U=[t_1,t_2]$ ($t_1<t_2$, 
for example $t_2=t_1+T$).
There is a natural number $M$ such that for each
solution of the MHR system the number of resets
of this solution for $t\in U$ is at most $M$.   
\end{proposition} 

\begin{proof}
Since, by Proposition~\ref{prop:est}, each reset requires at least 
$\delta>0$ units of time, only finitely many resets 
can occur in any finite time interval and the number of resets 
can be estimated uniformly by $M=\lfloor(t_2-t_1)/\delta\rfloor+1$. 
\end{proof}

\begin{corollary}
Define 
\[
S_n=
\{z\in\mathbb{P}\mid 
\text{$\phi(t;0,z)$ has exactly $n$ spikes for $t\in[0,T]$}\}.
\]
There is a nonnegative integer $n_{\max}$ such that
\[
\mathbb{P}=\bigcup_{i=0}^{\quad i=n_{\max}}S_i
\quad\text{and}\quad
S_i\cap S_j=\emptyset\;\;\text{for $i\neq j$,}
\]
so $S_0,S_1,\dotsc,S_{n_{\max}}$ form
a finite partition of $\mathbb{P}$.
\end{corollary}

\subsection{Gradation in the possible number of spikes fired 
with respect to initial conditions 
in the MHR system}\label{subsec:gradation}

While modeling neuron's activity, one of the fundamental issues 
is to distinguish whether a chosen model is able to model \emph{tonic} 
or \emph{phasic} spiking (or both). Tonic neurons fire infinitely many 
spikes (when the amplitude of the injected current is large enough)
whereas phasic spiking means that the neuron fires at most finitely 
many spikes and then becomes silent even if the same stimulus is kept.  

Figure~\ref{fig:spikes} reveals a clear organization of the
$(v,\theta)$ phase plane according to the total number 
of spikes generated by
trajectories over the infinite time horizon. The regions corresponding to
different spike counts are arranged in a structured manner, indicating
that the total number of spikes depends systematically on the initial
state rather than varying irregularly across the phase plane.

\begin{figure}[htbp]
    \includegraphics[scale =0.42,
    trim= 0mm 0mm 0mm 0mm]{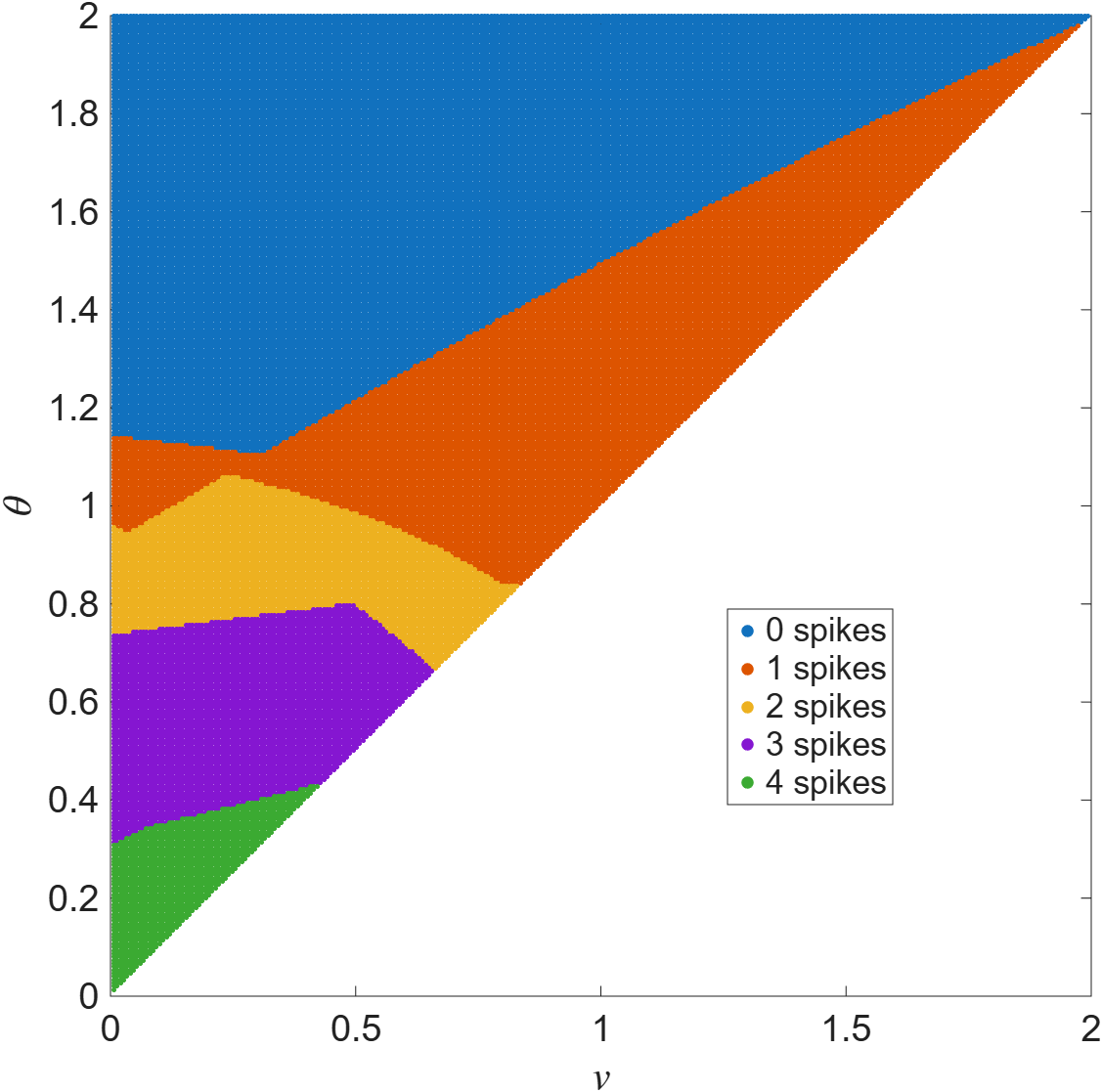}
    \caption{Finite number
    of spikes for all solutions in the MHR model
    with a unique periodic attractor, which is reset-free.
    Meaning of colors of initial values of 
    solutions in the phase space:
    blue  (no spikes), red (exactly one spike), 
    yellow (exactly two spikes), 
    purple (exactly three spikes),
    green (exactly four spikes).
    All other solutions (starting above the line $\theta=2$)
    produce $0$ spikes.
    Parameter values: $v_\mathrm{rest} = 0.1; 
    A = 2; T = 2; d = 0.31; \tau = 2;
    a = 0.08; b = 0.5; c = 0.53; 
    \Delta = 0.3; t_0=0$.
    The phase space
    $0\le v\le\theta\le2$ was discretized into a $300\times300$ 
    grid of initial conditions.
    }
    \label{fig:spikes}
\end{figure}

More precisely, Figure~\ref{fig:spikes} suggests some gradation 
in the possible number of fired spikes. 
Namely, a particularly pronounced trend 
is observed with respect to the initial
threshold value. As the initial threshold $\theta$ increases, the total
number of spikes shows a tendency not to increase, 
i.e., either stay the same or decrease,
whereas trajectories starting from
lower threshold values are capable of generating additional resets before
approaching the resting state. This numerical evidence naturally leads to
the following specific result and to a much more general, 
though as yet unproven, conjecture, which are formulated below.

\begin{proposition}\label{prop:numberofspikes}
Assume that $\theta_1>\theta_0>v_0\ge0$.
Let $\hat{x}(t)=\hat{x}(t;t_0,v_0,\theta_0)$ 
and $\hat{y}(t)=\hat{y}(t;t_0,v_0,\theta_1)$
denote the trajectories of the MHR system with initial conditions 
$(t_0,v_0,\theta_0)$ and $(t_0,v_0,\theta_1)$ respectively, i.e., 
$\hat{x}$ and $\hat{y}$ both start at $t_0$ with the same voltage value
$v_0$ but $\hat{y}$ starts above $\hat{x}$ in the $v$-$\theta$ plane.
Under the above assumptions,
if $\hat{x}(t)$ has no resets for $t\ge t_0$
then $\hat{y}(t)$ also has no resets for $t\ge t_0$.
\end{proposition}

Let us begin with the following simple observation, 
which we shall later use in the proof 
of Proposition~\ref{prop:numberofspikes}.

\begin{lemma}[Monotonicity Lemma]\label{lem:genineq}
Assume that ${x}(t;t_0,v_0,\theta_0)=(v_x(t),\theta_x(t))$ 
is an orbit of the non-hybrid NH system
\eqref{eq:mhr1}-\eqref{eq:mhr2}, 
starting at time $t_0$ from the initial point
$(v_0,\theta_0)$ and ${y}(t;t_0,v_1,\theta_1)=(v_y(t),\theta_y(t))$ 
is another orbit. Then
\begin{enumerate}
    \item if  $v_1>v_0$ ($v_1=v_0$) then $v_y(t)>v_x(t)$
    ($v_y(t)=v_x(t)$) for $t\ge t_0$,
    \item if $v_1\ge v_0$ and $\theta_1>\theta_0$ ($\theta_1\ge\theta_0$)
    then $\theta_y(t)>\theta_x(t)$ ($\theta_y(t)\ge \theta_x(t)$)
    for $t\ge t_0$.
\end{enumerate}
\end{lemma}  

\begin{proof}\emph{Ad (1).} From the general formula for the solution
(see Subsection~\ref{appsub:general})
\[
v_y(t)-v_x(t)=
\ue^{t_0-t}\big(v_1-v_0\big)
\quad\text{for $t\ge t_0$},
\]
which proves (1).\vspace{1mm}

\noindent\emph{Ad (2).} 
From the assumption $v_1\ge v_0$, the previous point (1) 
and the monotonicity of $f$, we obtain that
$f(v_y(s))-f(v_x(s))\ge0$ for $s\ge t_0$.
Now again from the general formula for the solutions
\[
\theta_y(t)-\theta_x(t)=
\ue^{(t_0-t)/\tau}
\Big((\theta_1-\theta_0)+\frac1\tau\int_{t_0}^t\!\! 
\big(f(v_y(s))-f(v_x(s))\big)\ue^{(s-t_0)/\tau}\,ds\Big)
\quad\text{for $t\ge t_0$},
\]  
which proves our claim.
\end{proof}

\begin{proof}[Proof of Proposition~\ref{prop:numberofspikes}]
Assume that $\hat{x}(t)=\hat{x}(t;t_0,v_0,\theta_0)=(\hat{v}_{x}(t;t_0,v_0,\theta_0), \hat{\theta}_{x}(t;t_0,v_0,\theta_0))$ has no resets, i.e.,
$\hat{x}(t)=x(t)=(v_x(t),\theta_x(t))$ 
(the hybrid and non-hybrid solutions with the same
initial conditions coincide).
Let $y(t)=(v_y(t),\theta_y(t))$ be the non-hybrid orbit
with initial point $y(t_0)=(v_0,\theta_1)$.
Since $\hat{x}(t)$ has no resets, 
by Monotonicity Lemma~\ref{lem:genineq}, 
\begin{equation}\label{eq:gwiazdka}
{\theta}_{y}(t)>{\theta}_{x}(t)\ge {v}_{x}(t)={v}_{y}(t)
\end{equation} 
and, in consequence, $\hat{y}(t)=y(t)$ also has no resets. 
\end{proof}

Finally, let us formulate, on the basis of numerical simulations, 
a conjecture, which generalizes Proposition~\ref{prop:numberofspikes}
and, roughly speaking, states that higher initial threshold
values correspond to equal or smaller spike counts.
In other words, we expect some monotonicity of the number of spikes 
fired with respect to the initial value of threshold variable $\theta$.
However, for now it is unclear for us whether this new formulation 
of the conjecture may require some additional assumptions such as
(i) $0\leq v_0<v_R+A$ and/or (ii) $f(v)>v$ (plot $f$ above diagonal).

\begin{conjecture}\label{conj:numberofspikes}
Let $n=0,1,2,\dotsc,\infty$.
Under the assumptions of Proposition~\ref{prop:numberofspikes},
if $\hat{x}(t)$ has exactly $n$ resets for $t\ge t_0$ then $\hat{y}(t)$
has at most $n$ resets for $t\ge t_0$,
i.e., for fixed $v_0$ and $t_0$ the number of spikes is a non-increasing
function of the initial value of $\theta$.
\end{conjecture}

\section{Discussion and conclusions}
\label{sec:discussion}

In this work, we studied a two-dimensional hybrid non-autonomous
neuron model with a voltage-dependent threshold, periodic external forcing,
and a discontinuous spike-and-reset mechanism. The model belongs to the
class introduced by Meng, Huguet, and Rinzel to investigate type~III
excitability, slope sensitivity, and coincidence detection
\cite{MHR2012}. Despite the relatively simple triangular structure of the
continuous subthreshold equations, the reset rule produces a substantially
richer range of dynamic responses. Below, we place the analytical and
numerical results obtained above within a common dynamical framework, while
carefully distinguishing rigorous conclusions, numerical observations, and
open conjectures.

The non-hybrid system provides the natural reference dynamics for the
hybrid MHR model. Its voltage component satisfies a scalar linear
non-autonomous equation with a unique attracting periodic solution. Once
this voltage response is substituted into the threshold equation, the
latter likewise admits a unique attracting periodic solution. Consequently,
the full non-hybrid system has a unique \(T\)-periodic orbit 
attracting every other solution. Thus, the continuous system 
asymptotically loses memory of both its initial state 
and its initial time, with all solutions synchronizing
to the same periodically forced response.

The reset mechanism changes this picture fundamentally. Rather than merely
superimposing spikes on the periodic dynamics of the non-hybrid system, it
acts as a selection mechanism determining whether the periodic orbit of the
continuous system remains admissible in the hybrid model. If the non-hybrid
periodic orbit remains entirely in the subthreshold region and therefore
triggers no reset, it becomes the unique reset-free periodic orbit of the
MHR system. If, instead, it crosses the threshold line and activates the
reset condition, it ceases to be an admissible reset-free solution, and all
trajectories must undergo infinitely many resets. Hence, the distinction
between reset-free and reset-including dynamics is directly determined by
the position of the unique non-hybrid periodic orbit relative to the
threshold boundary.

The creation of new recurrent responses by the reset rule is consistent
with the general behavior of hybrid integrate-and-fire models. In such
systems, continuous subthreshold evolution and discrete resets may combine
to produce spike patterns and bifurcation structures absent from the
underlying continuous flow 
\cite{TouboulBrette2009,CoombesThulWedgwood2012}. The present model 
provides a particularly transparent example: the non-hybrid dynamics 
has a unique global periodic attractor, 
whereas the hybrid system may support
reset-including periodic attractors and coexisting asymptotic responses.

Periodic solutions of the MHR system fall into two classes. A reset-free
periodic orbit evolves entirely according to the continuous equations. An
infinite-reset-including periodic orbit undergoes one or more resets during
each periodic cycle and therefore infinitely many resets over the full time
interval. Periodicity rules out a periodic orbit with only finitely many
resets. Our analytical results show that 
there can be at most one reset-free
periodic orbit. Moreover, the existence of 
any trajectory with only finitely
many resets is equivalent to the existence of this reset-free periodic
orbit, and every finite-reset trajectory converges to it.

A second important contribution is the construction of positively invariant
globally attracting regions (PIGARs). These regions localize all
asymptotically relevant dynamics within compact subsets of phase space and
provide a geometric framework for studying periodic orbits, reset
itineraries, and basins of attraction. The four possible geometries are
determined by the relative positions of the basic attracting rectangle
$P_{\mathrm{basic}}$, the threshold line $v=\theta$, 
and the post-spike reset
height. In particular, the quantity
$$
k=\Delta-f(v_{\mathrm{rest}}+A)+v_{\mathrm{rest}}+A
$$
determines the vertical extension needed to ensure that every reset maps a
trajectory back into the corresponding invariant polygon. The PIGAR
construction therefore gives a direct geometric interpretation of how the
input amplitude and the threshold increment $\Delta$ constrain the region
in which the long-term hybrid dynamics can occur.

The existence of a uniform positive lower bound on all interspike intervals is
also important from both mathematical and modeling perspectives. It
precludes the accumulation of infinitely many reset events in finite time
and hence rules out Zeno-type behavior in the MHR system. Consequently,
every trajectory undergoes only finitely many resets on any bounded time
interval. In particular, there is a uniform upper bound on the number of
spikes that any solution can generate during one forcing period. This makes
it possible, for the fixed initial forcing phase $t_0=0$, to partition the
phase space into finitely many sets
$$
\mathbb{P}=\bigcup_{j=0}^{n_{\max}}S_j, \qquad
S_i\cap S_j=\varnothing
\quad\text{for } i\neq j,
$$
where $S_j$ consists of initial conditions producing 
exactly $j$ spikes over
$[0,T]$. This finite spike-count partition may be treated
as a first step toward a symbolic
or combinatorial description of the hybrid stroboscopic dynamics.

The initial threshold value also appears to organize the transient spiking
response. The numerical classification shown 
in Figure~\ref{fig:spikes} suggests that,
for fixed initial voltage and forcing phase, 
increasing the initial threshold
does not increase the total number of spikes. Proposition
\ref{prop:numberofspikes} proves this statement rigorously 
in the zero-spike
case: if the trajectory starting from $(v_0,\theta_0)$ remains reset-free,
then any trajectory starting from the same voltage and a larger threshold
also remains reset-free. 

The analogous statement for trajectories undergoing one or more resets
remains open. The main difficulty is that 
the simple order relation preserved
by the non-hybrid flow need not survive the first reset. Although two
trajectories may initially have the same voltage and ordered threshold
values, they may reach the threshold line at different times, after which
their reset histories and forcing phases are no longer directly comparable.
Conjecture \ref{conj:numberofspikes} nevertheless proposes that the total
number of resets is a non-increasing function of the initial threshold. 

We next complement these analytical results with numerical observations on
reset-including periodic attractors. The simulations indicate that such
attractors occur repeatedly across different parameter values and are not
confined to an isolated numerical example. Their existence and stability,
however, have not yet been established analytically. The statements below
should therefore be interpreted as numerical observations rather than as a
rigorous classification.

\begin{figure}[htbp]
    \includegraphics[scale =0.44,
    trim= 0mm 0mm 0mm 0mm]{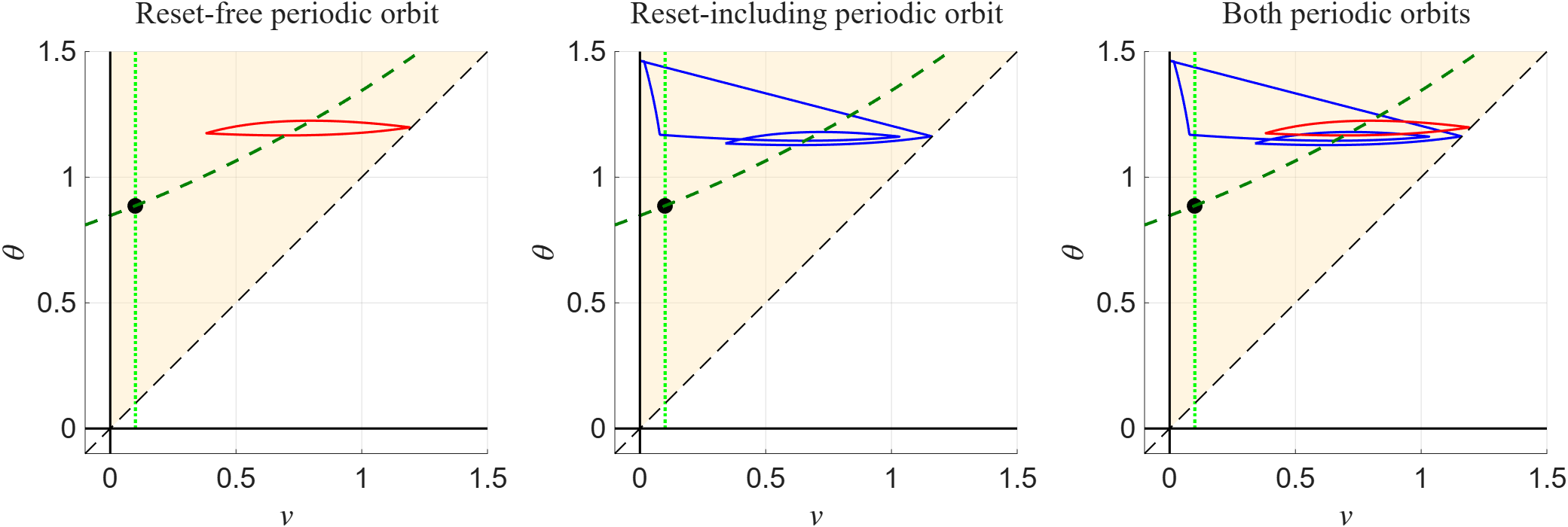}
    \caption{Bistability in the MHR model (pulse wave). Two periodic attractors for different initial conditions:
    reset-free $T$-periodic attractor (left),
    reset-including $2T$-periodic attractor (middle),
    their relative positions in phase space (right).
    Parameter values: $v_\mathrm{rest} = 0.1; 
    A = 2; T = 2; d = 0.32; \tau = 2;
    a = 0.08; b = 0.5; c = 0.53; 
    \Delta = 0.3$.
    }
    \label{fig:bistability}
\end{figure}

\begin{figure}[htbp]
    \includegraphics[scale =0.43,
    trim= 0mm 0mm 0mm 0mm]{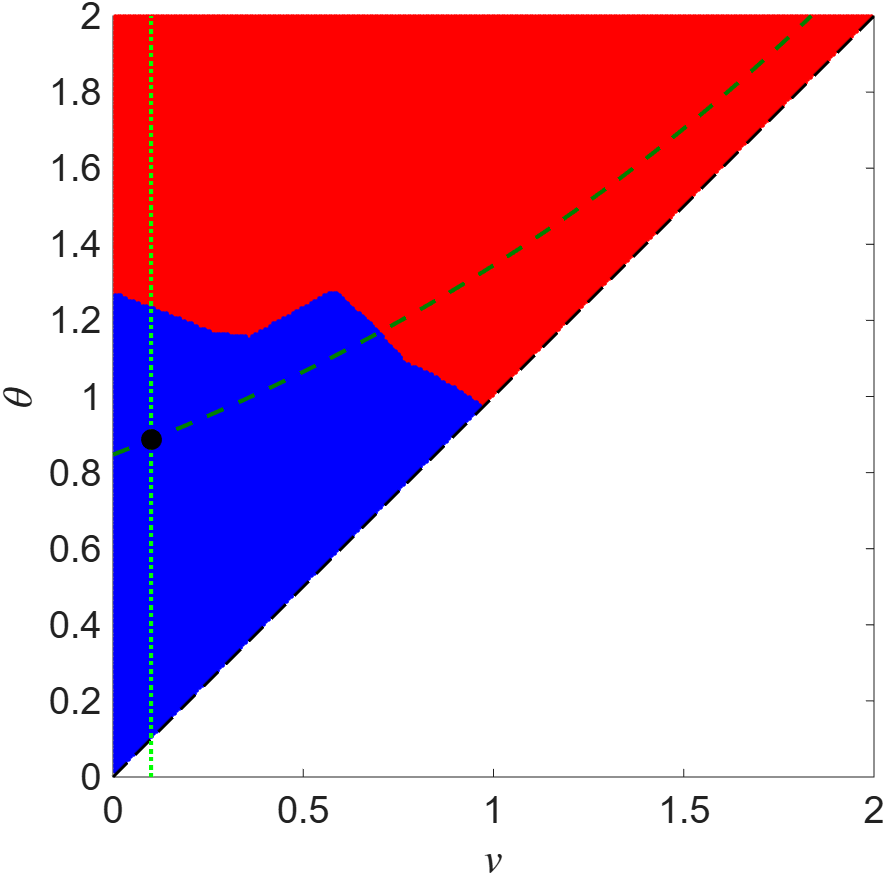}
    \caption{Basins of attraction
    in the phase space
    for bistability in the MHR system.
    Initial conditions are chosen
    for $t_0=0$. Orbits starting
    from blue (red) area tend to the
    reset-including (reset-free)
    periodic attractor and have infinite
    (finite) number of resets
    for $t\in[0,\infty)$.
    Parameter values as in 
    Figure~\ref{fig:bistability}. 
    The visible part of the phase space
    ($0\le v\le\theta\le2$) was discretized into a $300\times300$ grid
    of initial conditions with numerical computations restricted to this
    triangular region. 
    }
    \label{fig:basins}
\end{figure}

Figures \ref{fig:bistability} and \ref{fig:basins} 
illustrate the first form
of bistability detected in the MHR system. For the same parameter values, a
reset-free $T$-periodic attractor coexists with a reset-including
$2T$-periodic attractor. Initial conditions in one basin converge to a
non-spiking subthreshold response, whereas those in the other converge to a
persistent spiking response. Thus, the threshold-reset mechanism 
can produce qualitatively different asymptotic regimes 
without any change in the external forcing or model parameters.

This coexistence can be interpreted in terms of an effective refractory
memory encoded by the threshold variable. Dynamic spike thresholds have
been observed experimentally and related to neuronal sensitivity to the
temporal structure and rate of depolarization of the input
\cite{AzouzGray2000,AzouzGray2003}. Adaptive thresholds have also been
incorporated into reduced neuron models to reproduce a broad range of
spiking responses \cite{KobayashiEtAl2009}. 
In the present model, each spike
resets the voltage to zero and increases the threshold by $\Delta$. The
threshold variable therefore retains the effect of previous firing events,
while the current values of $(v,\theta,t \bmod T)$ are sufficient 
to determine the system’s future evolution.

The basin plot in Figure \ref{fig:basins} 
should be interpreted as a section
of the basins at the fixed forcing phase $t_0=0$. Because the model is
non-autonomous and periodically forced, changing the initial phase of the
input may deform the corresponding basin sections. The numerically observed
basin boundary is nontrivial and appears to be influenced by the threshold
line and changes in the number of resets. Identifying the invariant
structure organizing this boundary, possibly an unstable periodic orbit or
another unstable invariant set, is an interesting problem for future study.

\begin{figure}[htbp]
    \includegraphics[scale =0.5,
    trim= 0mm 0mm 0mm 0mm]{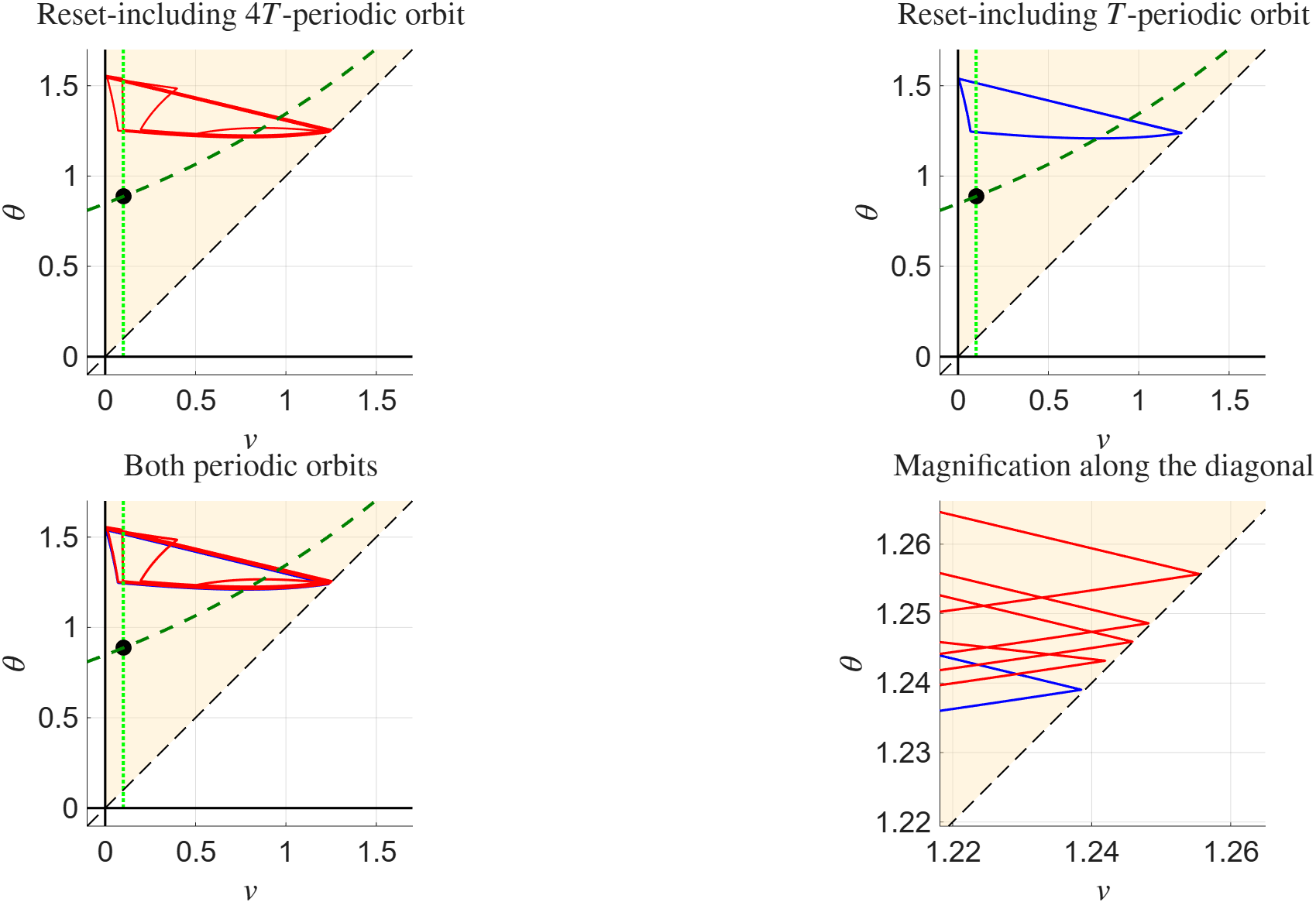}
    \caption{Another type of bistability in the MHR model 
    (pulse wave) with two reset-including periodic attractors. 
    Clockwise from upper left corner:
    reset-including $4T$-periodic attractor with
    $3$ reset points on the diagonal,
    reset-including $T$-periodic attractor
    with $1$ reset point,
    magnification along the diagonal showing 
    the relative positions of two
    attractors in the phase space.
    Parameter values: $v_\mathrm{rest} = 0.1; 
    A = 2; T = 2; d = 0.43; \tau = 2;
    a = 0.08; b = 0.5; c = 0.53; 
    \Delta = 0.3$.
    }
    \label{fig:bistability2}
\end{figure}

\begin{figure}[htbp]
    \includegraphics[scale =0.43,
    trim= 0mm 0mm 0mm 0mm]{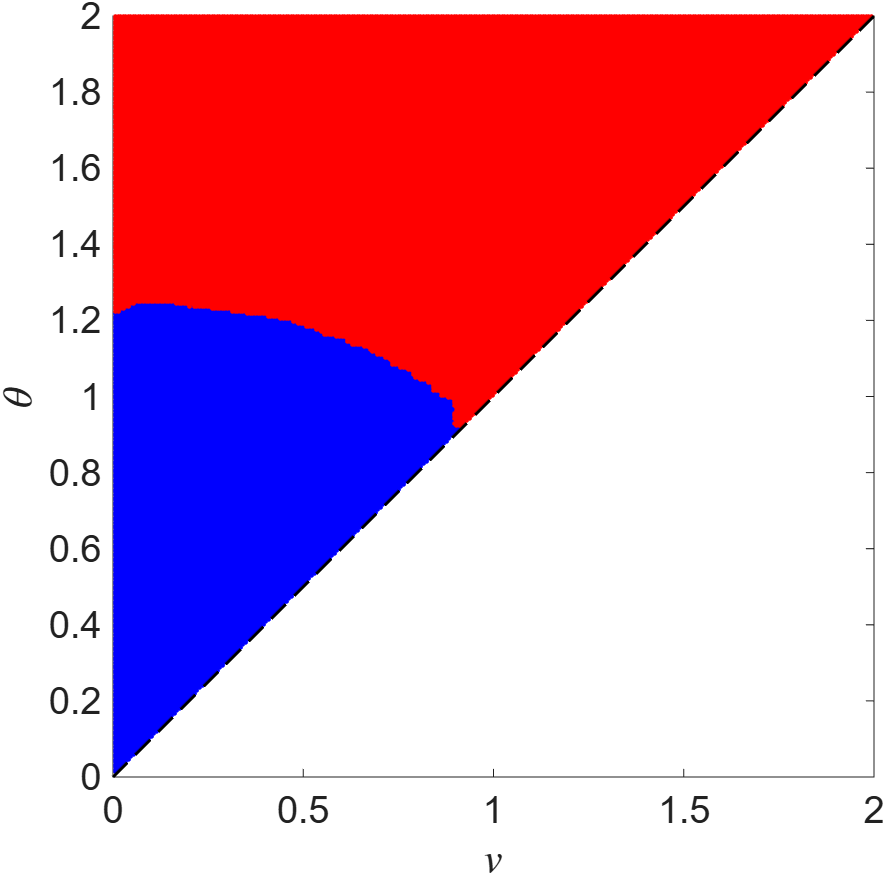}
    \caption{Basins of attraction
    for second type of bistability 
    in the MHR system.
    Initial conditions are chosen
    for $t_0=0$. Orbits starting
    from blue (red) area tend to the
    reset-including $T$-periodic 
    ($4T$-periodic) attractor.
    Parameter values as in 
    Figure~\ref{fig:bistability2}. 
    The dynamics was computed for $300\times300$
    grid of initial conditions, with computations 
    performed only for 
    points above the diagonal with $v,\theta\in[0,2]$. 
    }
    \label{fig:basins2}
\end{figure}

A second form of bistability is shown in Figures \ref{fig:bistability2} and
\ref{fig:basins2}. Here, both attracting solutions are reset-including but
have different temporal organizations. One attractor is $T$-periodic and
undergoes one reset per forcing period, 
whereas the other has minimal period
$4T$ and undergoes three resets over its periodic cycle. Thus, the same
external stimulus can support two distinct stable firing rhythms. Their
basins separate initial states leading to different 
long-term spike patterns,
rather than spiking from non-spiking behavior.

The coexistence of two reset-including attractors shows that the reset rule
can generate more than one recurrent itinerary through phase space. Each
itinerary is determined by successive intervals of continuous evolution,
threshold crossings, and post-spike jumps. In a periodically forced hybrid
system, different itineraries may close after different numbers of forcing
periods, producing attracting solutions with different minimal periods and
firing ratios. The coexistence observed here is consistent with the general
capacity of two-dimensional hybrid integrate-and-fire models to support
distinct recurrent spike patterns generated by different reset itineraries
\cite{TouboulBrette2009,CoombesThulWedgwood2012}. In the present system,
however, the reset-free/reset-including and reset-including/reset-including
forms of bistability arise from the specific interaction among periodic
forcing, a voltage-dependent threshold, and the spike-triggered threshold
increment.

Precise terminology is important when interpreting these numerical results.
For the parameter sets represented in Figures
\ref{fig:bistability}-\ref{fig:basins2}, the computations reveal two
coexisting \emph{stable periodic attractors}. They do not establish that
these are the only periodic orbits of the system. In particular, the
simulations do not exclude unstable periodic orbits, additional attractors
with very small basins, or other invariant sets. 
Thus, the numerical results
establish bistability between the two detected attractors, 
but not the exact number of periodic orbits in the system.

\begin{figure}[htbp]
    \includegraphics[scale =0.44,
    trim= 0mm 0mm 0mm 0mm]{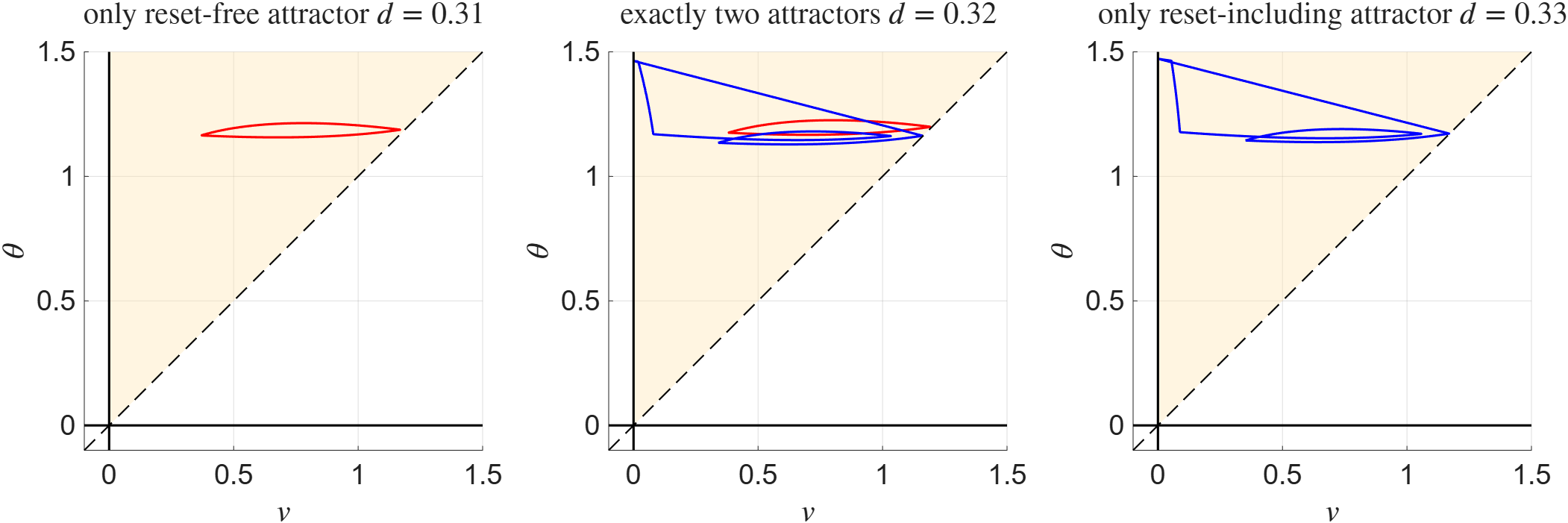}
    \caption{Numerically observed transition between periodic attracting
    regimes in the MHR model (pulse wave) with respect to $d$.
    Left panel: for $d=0.31$, the numerical simulations detect 
    a reset-free $T$-periodic attractor.
    Middle panel: for $d=0.32$, the numerical simulations reveal 
    the coexistence of two stable periodic attractors, 
    one reset-free and the other reset-including. 
    Right panel: for $d=0.33$, the numerical simulations detect 
    a reset-including $2T$-periodic attractor.
    Other common parameter values: 
    $v_\mathrm{rest} = 0.1; 
    A = 2; T = 2; \tau = 2;
    a = 0.08; b = 0.5; c = 0.53; 
    \Delta = 0.3$.
    }
    \label{fig:bifurcation}
\end{figure}

Figures \ref{fig:bifurcation} and \ref{fig:patterns} 
show how the attracting
periodic response changes as the duty cycle $d$ 
of the pulse input is varied.
For $d=0.31$, the observed attractor is a reset-free $T$-periodic solution.
At $d=0.32$, a reset-free periodic attractor coexists 
with a reset-including $2T$-periodic attractor. 
For $d=0.33$, the numerical trajectories converge
to the reset-including $2T$-periodic attractor. 
These computations therefore
indicate a transition from reset-free to reset-including dynamics through
an intermediate bistable regime.

At this stage, we refer to this phenomenon as a numerically observed
transition rather than a fully identified bifurcation. A rigorous
bifurcation description would require continuation of the periodic
solutions, stability information, and analysis of the event at which an
orbit first touches or crosses the threshold line. In related periodically
forced integrate-and-fire systems, such transitions have been associated
with smooth and non-smooth grazing, gluing, and border-collision mechanisms
\cite{GH2019,GKC2014,CoombesThulWedgwood2012}. Similar non-smooth mechanisms 
are natural candidates for the present model.

To describe the different locked firing patterns, suppose that a periodic
solution has minimal period $qT$ and undergoes $p$ resets 
over this interval.
We define its firing ratio by
$$
\rho=\frac{p}{q}.
$$
This ratio is the average number of spikes per forcing period along the
periodic orbit. In the numerical parameter sweep shown in Figure
\ref{fig:patterns}, we observe the sequence
$$
\rho=\frac{0}{1},\qquad
\frac{1}{2},\qquad
\frac{2}{3},\qquad
\frac{3}{4},\qquad
\frac{1}{1}.
$$
The intermediate ratios follow
$$
\rho=\frac{k}{k+1},\qquad k=0,1,2,3,
$$
before the transition to one-to-one locking. The ratio $1/1$ is not the
next term of the sequence $k/(k+1)$; rather, it represents the final
one-to-one locked regime observed in this parameter sweep.

\begin{figure}[htbp]
    \includegraphics[scale =0.408,
    trim= 2mm 0mm 0mm 0mm]{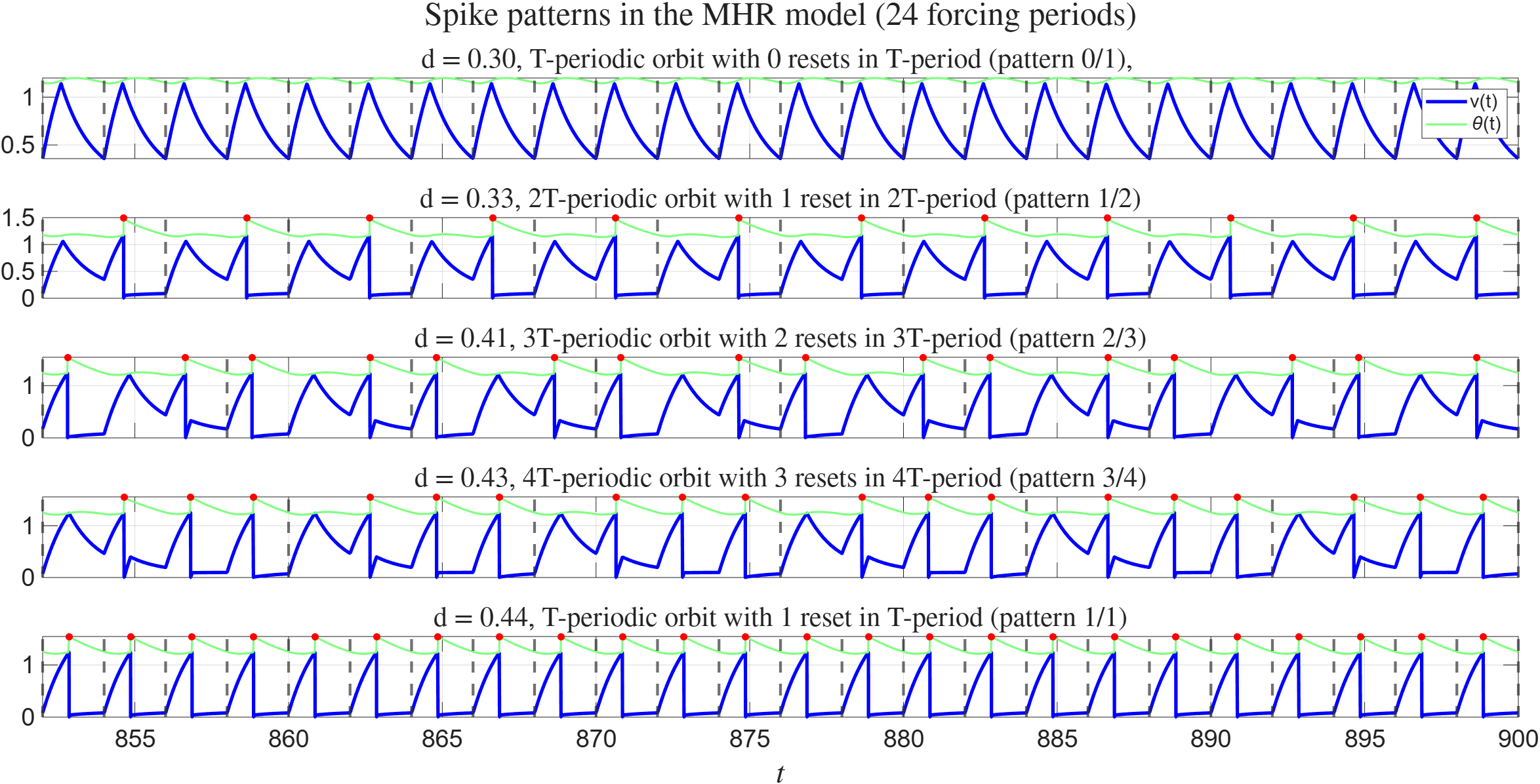}
    \caption{Five different spike patterns
    of periodic orbits in the MHR system
    with respect to the parameter $d$.
    Other common parameter values
    as in Figure~\ref{fig:bifurcation}.
    }
    \label{fig:patterns}
\end{figure}

This organization is reminiscent of mode locking in periodically forced
integrate-and-fire models \cite{Keener1981,Coombes1999,COS2001}. 
It also suggests a period-incrementing or 
period-adding-like structure in the hybrid
stroboscopic map, similar to those described for other periodically driven
spiking systems \cite{GKC2014,GK2015,GAK2017}. 
Nevertheless, the numerical sequence
reported here is insufficient to establish a complete Farey organization or
a rigorous rotation-number theory. Such conclusions would require a broader
parameter study and a precise analysis of the stroboscopic map.

In the MHR system, the time-$T$ map is expected to be smooth on regions of
initial conditions sharing the same number and order of resets, with smooth
branches corresponding to fixed reset itineraries. Across the boundaries of
these regions, the reset itinerary changes, and the stroboscopic map may
become discontinuous or only piecewise smooth. 
Transitions between different
firing ratios may therefore be mediated by grazing 
or border-collision events
\cite{GH2019,GKC2014}. A geometric analysis of these transitions, together with
the possible construction of an appropriate rotation number or symbolic
firing sequence, will be pursued in subsequent work.

The role of the duty cycle can be understood as a competition between
excitation and refractoriness. Increasing $d$ lengthens the active phase of
the external current and thereby tends to promote threshold crossings. Each
crossing, however, resets the voltage to zero and raises the threshold by
$\Delta$, delaying the next possible spike. Thus, increasing $d$ need not
produce a simple monotone increase in spike count on every time scale.
Instead, stable firing patterns arise from a balance between the excitatory
action of the input and the refractory effect of the threshold increment.
This mechanism naturally allows one spike per period, one spike every two
periods, or several spikes distributed over a longer periodic cycle.

For each pulse-forced parameter set explored numerically in this work, the
computed long-time dynamics exhibited either one or two stable periodic
attractors. We found no parameter values with more than two stable periodic
attractors. When two attractors coexisted, 
either one was reset-free and the
other reset-including, or both were reset-including. In the latter cases
observed so far, the two attracting orbits had different 
minimal periods and firing ratios.

These observations motivate the following restricted conjectural picture:
for the pulse-forced MHR system, at least within the parameter regime
investigated here, the stroboscopic dynamics admits at most two stable
periodic attractors. If two stable attractors coexist, at most one can be
reset-free, as follows from the uniqueness of the periodic orbit of the
non-hybrid system. When two reset-including attractors coexist, our
simulations suggest that they have different minimal periods and firing
ratios. These statements should not yet be interpreted as global assertions
for all admissible parameters or periodic inputs.

Recall that for any constant value $A$ of the input $I(t)=A$, 
the non-hybrid MHR model \eqref{eq:mhr1}-\eqref{eq:mhr2} 
has a unique stable equilibrium point. 
This equilibrium point is subthreshold 
if $f(v_\mathrm{rest}+A)>v_\mathrm{rest}+A$. 
Therefore, with an additional assumption that $f(v)>v$ 
for all $v\geq 0$, the model was introduced in \cite{MHR2012} 
as an idealized model of the Type III excitability. 
Note that in our work we assumed that 
$f(v_\mathrm{rest})>v_\mathrm{rest}$ (i.e. the stable equilibrium 
point is located above the diagonal for $A=0$) but in general 
we do not assume that $f(v)>v$ or even that
$f(v_\mathrm{rest}+A)>v_\mathrm{rest}+A$, 
therefore our results persist beyond Type III excitability. 
Nevertheless, in particular in parameters regime used 
in Figures~\ref{fig:bistability}-\ref{fig:basins2}, 
the $\theta$-nullcline $\theta=f(v)$ remains above the diagonal 
$\theta=v$ for all values of $v$ which means the system has 
a stable subthreshold equilibrium point for all 
constant positive values of the input $I$ and therefore 
can be referred to as Type III model. For these parameters values 
the asymptotic behavior of the model (and thus the spike pattern fired)
for the same periodic stimulus $I=I(t)$ depends strongly on the initial
conditions as follows from various types of bistability illustrated in 
Figures \ref{fig:bistability} and \ref{fig:bistability2}. In this case,
a small noise or instantaneous voltage kick of minimal size applied to 
a deterministic trajectory can cause switching between different firing
modes. On the other hand, the model displays strong selectivity with
respect to the input parameters. As visualized in Figures
\ref{fig:bifurcation} and \ref{fig:patterns}, a sort of phase locking
occurs with a specific mode-locking pattern strongly tied to the duration
of the non-zero input in periodic (pulse-wave) stimulus. These properties
are more likely to occur for Type III neurons than in other excitability
types (\cite{MHR2012}).

We also note that the multistability phenomena and coexistence 
of periodic attractors of different periods was reported in~\cite{GH2019}
for similar parameters choice as in 
Figures~\ref{fig:bistability}-\ref{fig:basins2} 
(which also follows parameter regimes studied in
\cite{MHR2012}). The authors analyze Farey neighbors and 
illustrate period-adding phenomena using advanced numerical 
techniques but rigorous results concerning structure and transitions 
of periodic attractors still remain a challenge for future work.

Therefore, the rigorous existence of reset-including periodic orbits 
is one of the main open problems raised by this work. 
A possible strategy is to construct a
hybrid Poincar\'e or stroboscopic map on 
a compact invariant region in which
the number and order of resets are fixed. Under suitable transversality
assumptions, the reset times should depend continuously, and possibly
smoothly, on the initial condition. One could then seek to establish the
existence and stability of a periodic orbit through a fixed-point argument,
a contraction estimate, or an appropriate topological method. The PIGARs
constructed here provide natural compact domains on which 
such maps could be studied.

A systematic continuation analysis is also needed to determine how
reset-including periodic attractors are created and destroyed as the duty
cycle $d$, amplitude $A$, forcing period $T$, or reset height $\Delta$ are
varied. Such an analysis should distinguish smooth bifurcations occurring
within a region of fixed reset itinerary 
from non-smooth bifurcations arising
when an orbit becomes tangent to the threshold line or changes 
its number of resets. It would also clarify whether 
the observed sequence of firing ratios belongs to 
a broader period-adding or period-incrementing structure \cite{GAK2017}.

Further investigation of the basins of attraction would complement this
analysis. In particular, it would be useful to determine the regularity and
geometry of their boundaries, their dependence 
on the initial forcing phase,
and their sensitivity to model parameters. From a neuronal perspective,
these basin boundaries quantify how small changes in the initial voltage or
threshold may switch the long-term response between non-spiking and
persistent-spiking regimes, or between distinct spiking rhythms.

In summary, the model studied here demonstrates how a simple
two-dimensional periodically forced system can acquire rich dynamics when
equipped with a dynamic threshold and a reset mechanism. The non-hybrid
system has a globally simple structure with a unique attracting periodic
response. By contrast, the hybrid system 
may exhibit transient or persistent
spiking, coexistence of stable periodic responses, nontrivial basins of
attraction, and locked firing patterns with different minimal periods and
firing ratios. The analytical results provide a rigorous framework for
localizing the dynamics and controlling reset events, whereas the numerical
observations reveal phenomena genuinely induced by the hybrid structure.
Together, these results show that dynamic thresholds 
and reset mechanisms can generate a complex organization of 
firing patterns even in low-dimensional neuron models.

\appendix

\section{Direct formulas for solutions}\label{app:A}
Since the equations under consideration 
disregarding the reset effect are quite simple in nature 
(linear with respect to main variables), 
we are able to give direct formulas for their solutions
by first solving the first equation, which does not depend on 
$\theta$, and then solving the second one with the given solution 
of the first one inserted into the second.
Knowledge of explicit forms of solutions enables 
both their estimation and the analysis 
of their asymptotic behavior

\subsection{General case}\label{appsub:general}
Observe that the NH system 
\eqref{eq:mhr1}-\eqref{eq:mhr2} (non-autonomous ODE) 
can be solved by quadratures, i.e.,
expressed in terms of integrals. 
Namely, from the first equation, 
assuming the initial condition is  $v(t_0)=v_0$, we obtain
\begin{equation}\label{eq:v_general}
v(t)=
v(t;t_0,v_0)=
\ue^{t_0-t}\big(v_0+H(t)\big),
\quad\text{where}\quad
H(t)=\int_{t_0}^t\!\!\big(v_\mathrm{rest}+I(s)\big)\ue^{s-t_0}\,ds
\end{equation}
or, equivalently,
\begin{equation}\label{eq:v_general2}
v(t)=v_\mathrm{rest}+
(v_0-v_\mathrm{rest})\ue^{t_0-t}+\ue^{-t}\int_{t_0}^t\!\!\! I(s)\ue^{s}ds.
\end{equation}
Similarly, from the second equation, 
assuming $\theta(t_0)=\theta_0$, we get
\begin{equation}\label{eq:theta_general}
\theta(t)=
\ue^{(t_0-t)/\tau}\big(\theta_0+G(t)\big),
\quad\text{where}\quad
G(t)=\frac1\tau\int_{t_0}^t\!\! f\big(v(s)\big)\ue^{(s-t_0)/\tau}\,ds
\end{equation}
and
\begin{equation}\label{eq:f_general}
f\big(v(s)\big)=
a+\exp(b(v(s)-c))=
a+\exp\left[b\bigg(\ue^{t_0-s}\Big(v_0+\int_{t_0}^s\!\!\big(v_\mathrm{rest}+I(u)\big)
\ue^{u-t_0}\,du\Big)-c\bigg)\right].
\end{equation}
Of course, $v(t)$ and $\theta(t)$ are the solutions of 
the NH system, i.e., the non-autonomous 
ODE \eqref{eq:mhr1}-\eqref{eq:mhr2}
without reset conditions. However, if we take into account the reset, 
we need to modify the initial conditions after the reset, 
but the form of the solutions remains the same.

\subsection{Pulse wave case}
Note that if we take the input $I(t)$ as a pulse wave 
function~\eqref{eq:inputSquare} then
we may treat the NH system as a switched autonomous ODE system,
i.e.,  ODEs and a rule (a switching signal, $I=0$ or $I=A$) 
that determines which of the two equations 
is active at any given moment.
Namely, we are able to compute the solution of the system, 
distinguishing two cases: first, when $t\in(nT, nT+dT]$ 
and second, when $t\in(nT+dT, (n+1)T]$
according to the formula for the pulse wave input.
Namely, the solution $v(t)$ for $t\in[0,\infty)$ 
is given inductively (with respect to $n=0,1,2,\dotsc$ and 
division of the interval $(nT, (n+1)T]$  by point $t=nT+dT$) as
\[
v(0)=v(0T)=v_0 \text{ and }
v(t;0,v_0)=
\begin{cases}
v_\mathrm{rest}+A+(v(nT)-v_\mathrm{rest}-A)\ue^{nT-t} 
& \text{if $t\in(nT, nT+dT]$},\\
v_\mathrm{rest}+(v(nT+dT)-v_\mathrm{rest})\ue^{nT+dT-t} 
& \text{if $t\in(nT+dT, (n+1)T]$}.
\end{cases}
\]
In the above formulas $v(nT)$ (resp. $v(nT+dT)$)
is the initial data for each interval of $(nT, nT+dT]$ 
(resp. $(nT+dT, (n+1)T]$).
Observe that, by definition, the solution $v(t)$ 
is continuous and even piecewise differentiable. 
In particular, it is continuous 
both at $t=nT$ and $t=nT+dT$ (the two formulas agree
on these points). However, as it is easy to check,
$v(t)$ is not differentiable at these points
although it has one-sided derivatives.

Similarly, we obtain the formula for $\theta(t)$ 
starting from $\theta(0)=\theta(0T)=\theta_0$
and taking
{\small
\[
\theta(t)=
\begin{cases}
\theta(nT)\ue^{(nT-t)/\tau}+a(1-\ue^{(nT-t)/\tau})
+\frac{\ue^{(nT-t)/\tau}}{\tau}\ue^{b(v_\mathrm{rest}+A-c)}I_1(t) & 
\!\!\!\text{if $t\!\in\!(nT,nT+dT]$},\\
\theta(nT+dT)\ue^{(nT+dT-t)/\tau}+a(1-\ue^{(nT+dT-t)/\tau})
+\frac{\ue^{(nT+dT-t)/\tau}}{\tau}\ue^{b(v_\mathrm{rest}-c)}I_2(t) & 
\!\!\!\text{if $t\!\in\!(nT+dT,(n+1)T]$,}
\end{cases}
\]}
where
\begin{align*}
I_1(t)=&\int_{nT}^t\!\ue^{b(v(nT)-v_\mathrm{rest}-A)\ue^{nT-s}}
\ue^{(s-nT)/\tau}\, ds,\\
I_2(t)=&\int_{nT+dT}^t\!\ue^{b(v(nT+dT)-v_\mathrm{rest})\ue^{nT+dT-s}}
\ue^{(s-nT-dT)/\tau}\, ds.
\end{align*}

\subsection{Half-wave-rectified wave case}
Since for this case, we can rewrite the input function $I(t)$ as
\[
I(t)=\begin{cases}
  A\sin(2\pi \omega t)  & \text{if $t\in\left(nT,nT+\frac{T}{2}\right]$}, \\
  0 & \text{if $t\in\left(nT+\frac{T}{2},(n+1)T\right]$},
\end{cases}
\]
the NH system may now be treated as a switched ODE system 
consisting of both non-autonomous and autonomous subsystems, 
which, of course, is non-autonomous as a whole.
With this input the solution $v(t)$ is given by
{\small
\[
v(t)=
\begin{cases}
v_\mathrm{rest}+\big(v(nT)-v_\mathrm{rest}+\frac{2\pi\omega A}{1+4\pi^2\omega^2}\big)\ue^{nT-t}
+\frac{A(\sin(2\pi\omega t)-2\pi\omega \cos(2\pi\omega t))}
{1+4\pi^2\omega^2},
&\text{if $t\in(nT, nT+\frac{T}{2}]$}\\
v_\mathrm{rest}+\big(v\big(nT+\frac{T}{2}\big)-
v_\mathrm{rest}\big)\ue^{nT+\frac{T}{2}-t},
&\text{if $t\in(nT+\frac{T}{2}, (n+1)T]$}.
\end{cases}
\]}
Analogously, we can calculate $\theta(t)$ to obtain
{\small
\[
\theta(t)=
\begin{cases}
\theta(nT)e^{-(t-nT)/\tau}+a(1-\ue^{-(t-nT)/\tau})
+\frac{\ue^{-(t-nT)/\tau}}{\tau}\ue^{b(v_\mathrm{rest}-c)}I_3(t), 
&\!\!\!\!t\in(nT, nT+\frac{T}{2}]\\
\theta\!\left(nT+\frac{T}{2}\right)\!\ue^{-(t-nT-T/2)/\tau}
\!+a(1\!-\!\ue^{-(t-nT-T/2)/\tau})
+\frac{\ue^{-(t-nT-T/2)/\tau}}{\tau}\ue^{b(v_\mathrm{rest}-c)}I_4(t),
&\!\!\!\!t\in(nT+\frac{T}{2}, (n+1)T]
\end{cases}
\]}
with
$$
I_3(t)=\int_{nT}^t\ue^{b(v(nT)-v_\mathrm{rest})
\ue^{-(s-nT)}}\ue^{m(s)}\ue^{(s-nT)/\tau}ds,
$$
where
$$
m(s)=\frac{Ab}{1+4\pi^2\omega^2}\left(2\pi\omega 
\ue^{-(s-nT)}+\sin(2\pi\omega s)
-2\pi\omega\cos(2\pi\omega s)\right),
$$
and
$$
I_4(t)=\int_{nT+T/2}^t
\ue^{b(v\left(nT+\frac{T}{2}\right)-v_\mathrm{rest})
\ue^{-\left(s-nT-\frac{T}{2}\right)}}
\ue^{\left(s-nT-\frac{T}{2}\right)/\tau}\ ds.
$$

\section*{CRediT authorship contribution statement}
\textbf{Piotr Bart{\l}omiejczyk:}  Writing – original draft, 
Writing – review \& editing, Methodology, Investigation, Formal analysis,
Conceptualization, Validation, Visualization, Software;
\textbf{Juan Belmonte-Beitia:}  Writing – original draft, 
Writing – review \& editing, Methodology, Investigation, 
Formal analysis, Validation, Funding acquisition;
\textbf{Justyna Signerska-Rynkowska:}  Writing – original draft, 
Writing – review \& editing, Investigation, Formal analysis, 
Conceptualization, Validation, Funding acquisition.

\section*{Declaration of competing interest} 
The authors declare that they have no known competing financial
interests or personal relationships that could have appeared 
to influence the work reported in this paper.

\section*{Acknowledgements}  
 Justyna Signerska-Rynkowska was supported by NCN 
(National Science Centre, Poland) grant no.~2019/35/D/ST1/02253 
and by the Dioscuri program initiated by the Max Planck Society,
jointly managed with the National Science Centre (Poland), 
and mutually funded by the Polish Ministry of Science 
and Higher Education and the German Federal Ministry 
of Research, Technology and Space.

Juan Belmonte-Beitia was partially supported by project
PID2024-155384OB-C21, funded by Ministerio de Ciencia e 
Innovación/Agencia Estatal de Investigación, 
Spain (doi:10.13039\-/501100011033) and European Regional 
Development Fund (ERDF A way of making Europe). 
This research is part of the research project 
SBPLY/23/180225/000041, funded by the European Union through 
the European Regional Development Fund (ERDF) and 
by the Regional Government of Castilla-La Mancha (JCCM) 
through the INNOCAM programme.

\section*{Data availability} 
No data was used for the research described in the article.

\bibliographystyle{elsarticle-num}
\bibliography{referencesBBBSR}

\end{document}